\documentclass{birkjour}
\usepackage{amssymb,bbm}
\newtheorem{thm}{Theorem}[section]
 \newtheorem{cor}[thm]{Corollary}
 \newtheorem{lem}[thm]{Lemma}
 
 \theoremstyle{definition}
 
 \theoremstyle{remark}

 \numberwithin{equation}{section}
\begin{document}

\title[Perturbation of singular vectors]
 {Perturbation of linear forms of singular\\ vectors under Gaussian noise}
\author[Koltchinskii]{Vladimir Koltchinskii}

\address{%
School of Mathematics\\
Georgia Institute of Technology\\
Atlanta, GA 30332\\
USA}

\email{vlad@math.gatech.edu}

\thanks{The research of Vladimir Koltchinskii was supported in part by NSF Grants DMS-1207808 and CCF-1415498.
Dong Xia's research was partially supported by NSF Grant DMS-1207808.}

\author[Xia]{Dong Xia}
\address{School of Mathematics\\
Georgia Institute of Technology\\
Atlanta, GA 30332\br
USA}
\email{dxia7@math.gatech.edu}



\date{January 1, 2004}

\begin{abstract}
Let $A\in\mathbb{R}^{m\times n}$ be a matrix of rank $r$ with singular value decomposition (SVD) $A=\sum_{k=1}^r\sigma_k (u_k\otimes v_k),$ 
where $\{\sigma_k, k=1,\ldots,r\}$ are singular values of $A$ (arranged in a non-increasing order) and 
$u_k\in {\mathbb R}^m, v_k\in {\mathbb R}^n, k=1,\ldots, r$ are the corresponding 
left and right orthonormal singular vectors. 
Let $\tilde{A}=A+X$ be a noisy observation of $A,$ where $X\in\mathbb{R}^{m\times n}$ is a random matrix with i.i.d. Gaussian entries, 
$X_{ij}\sim\mathcal{N}(0,\tau^2),$ and consider its SVD $\tilde{A}=\sum_{k=1}^{m\wedge n}\tilde{\sigma}_k(\tilde{u}_k\otimes\tilde{v}_k)$
with singular values $\tilde{\sigma}_1\geq\ldots\geq\tilde{\sigma}_{m\wedge n}$ and singular vectors $\tilde{u}_k,\tilde{v}_k,k=1,\ldots, m\wedge n.$

The goal of this paper is to develop sharp concentration bounds for linear forms $\langle \tilde u_k,x\rangle, x\in {\mathbb R}^m$ and 
$\langle \tilde v_k,y\rangle, y\in {\mathbb R}^n$ of the perturbed (empirical) singular vectors in the case when the singular values of $A$ are 
distinct and, more generally, concentration bounds for bilinear forms of projection operators associated with SVD.
In particular, the results imply upper bounds of the order $O\biggl(\sqrt{\frac{\log(m+n)}{m\vee n}}\biggr)$ (holding with a high probability) on 
$$\max_{1\leq i\leq m}\big|\big<\tilde{u}_k-\sqrt{1+b_k}u_k,e_i^m\big>\big|\ \ 
{\rm and} 
\ \ 
\max_{1\leq j\leq n}\big|\big<\tilde{v}_k-\sqrt{1+b_k}v_k,e_j^n\big>\big|,$$ 
where $b_k$ are properly chosen constants characterizing the bias
of empirical singular vectors $\tilde u_k, \tilde v_k$ 
and $\{e_i^m,i=1,\ldots,m\}, \{e_j^n,j=1,\ldots,n\}$ are the canonical bases of $\mathbb{R}^m, {\mathbb R}^n,$ respectively. 
\end{abstract}
\maketitle
\section{Introduction and main results}\label{introsec}

Analysis of perturbations of singular vectors of matrices under a random noise is of importance in a variety of areas including, for instance,  
digital signal processing, numerical linear algebra and spectral based methods of community detection in large networks 
(see \cite{kannan2009spectral}, \cite{stewart1998perturbation}, \cite{eisenstat1994relative}, \cite{li1998relative}, 
\cite{rohe2011spectral}, \cite{lei2014consistency}, \cite{jin2015fast}, \cite{huang2009spectral} and references therein).
Recently, random perturbations of singular vectors have been 
studied in Vu~\cite{vu2011singular}, Wang~\cite{wang2012singular}, 
O'Rourke et al.~\cite{o2013random}, Benaych-Georges and Nadakuditi~\cite{benaych2012singular}. 
However, up to our best knowledge, this paper proposes first sharp results concerning concentration 
of {\it the components} of singular vectors of randomly perturbed matrices. 
At the same time, there has been interest in the recent literature in so called ``delocalization'' properties of eigenvectors of random matrices, 
see Vershynin~\cite{rudelson2013delocalization}, Vu and Wang~\cite{vu2014random} and references therein. In this case, 
the ``information matrix'' $A$ is equal to zero, $\tilde A=X$ and, under certain regularity conditions, 
it is proved that the magnitudes of the components for the eigenvectors of $X$ (in the case of symmetric square matrix)
are of the order $O\big(\frac{\log(n)}{\sqrt{n}}\big)$ with a high probability. This is somewhat similar to the results 
on ``componentwise concentration'' of singular vectors of $\tilde A=A+X$ proved in this paper, but the analysis in the 
case when $A\neq 0$ is quite different (it relies on perturbation theory and on the condition that the gaps between 
the singular values are sufficiently large).\\

Later in this section, we provide a formal description of the problem studied in the current paper.
Before this, we introduce the notations that will be used throughout the paper. 
For nonnegative $K_1,K_2$, the notation $K_1\lesssim K_2$ (equivalently, $K_2 \gtrsim K_1$) means that there exists an absolute constant $C>0$ such that $K_1\leq CK_2;$
$K_1\asymp K_2$ is equivalent to $K_1\lesssim K_2$ and $K_2\lesssim K_1$ simultaneously. 
In the case when the constant $C$ might depend on $\gamma,$  we provide these symbols with subscript $\gamma:$ say, $K_1\lesssim_{\gamma} K_2.$ There will be many constants involved in the arguments that may evolve from line to line.

In what follows, $\langle \cdot, \cdot \rangle$ denotes the inner product of finite-dimensional Euclidean spaces. 
 For $N\geq 1,$ $e^N_j, j=1,\dots, N$ denotes the canonical basis of the space $\mathbb{R}^N.$
If $P$ is the orthogonal projector onto a subspace $L\subset\mathbb{R}^{N},$ then $P^{\perp}$ denotes the projector onto the orthogonal complement $L^{\perp}.$
With a minor abuse of notation, $\|\cdot\|$
denotes both the $l_2$-norm of vectors in finite-dimensional spaces and the operator norm of matrices (i.e., their largest singular value). 
The Hilbert-Schmidt norm of matrices is denoted by $\|\cdot\|_2$. 
Finally, $\|\cdot\|_{\infty}$ is adopted for the $l_{\infty}$-norm of vectors.

In what follows, $A'\in {\mathbb R}^{n\times m}$ denotes the  transpose of a matrix $A\in {\mathbb R}^{m\times n}.$
The following mapping $\Lambda: {\mathbb R}^{m\times n}\mapsto {\mathbb R}^{(m+n)\times (m+n)}$ will 
be frequently used:
$$
\Lambda(A):=\Big(\begin{array}{cc}0&A\\A'&0\end{array}\Big), A\in {\mathbb R}^{m\times n}.
$$
Note that the image $\Lambda(A)$ is a symmetric $(m+n)\times (m+n)$ matrix.

Vectors $u\in {\mathbb R}^m, v\in {\mathbb R}^n,$ etc. will be viewed as column vectors (or $m\times 1, n\times 1,$ etc matrices). 
For $u\in {\mathbb R}^m, v\in {\mathbb R}^n,$ denote by $u\otimes v$ the matrix $uv'\in {\mathbb R}^{m\times n}.$ 
In other words, $u\otimes v$ can be viewed as a linear transformation from ${\mathbb R}^n$ into ${\mathbb R}^m$ 
defined as follows: $(u\otimes v)x = u\langle v,x\rangle, x\in {\mathbb R}^n.$ 

Let $A\in\mathbb{R}^{m\times n}$ be an $m\times n$ matrix 
and let 
$$A=\sum_{i=1}^{m\wedge n}\sigma_i (u_i\otimes v_i)$$ 
be its singular value decomposition (SVD) with singular values $\sigma_1\geq\ldots\geq\sigma_{m\wedge n}\geq 0,$ orthonormal left singular vectors $u_1,\dots, u_{m\wedge n}\in {\mathbb R}^m$ and orthonormal 
right singular vectors $v_1,\dots, v_{m\wedge n}\in {\mathbb R}^n.$ 
If $A$ is of rank $\text{rank}(A)=r\leq m\wedge n,$ then $\sigma_i=0, i>r$ and the SVD can be written as 
$A=\sum_{i=1}^{r}\sigma_i (u_i\otimes v_i).$ 
Note that in the case when there are repeated singular values $\sigma_i,$ the singular vectors are not 
unique. In this case, let $\mu_1>\dots \mu_d>0$ with $d\leq r$ be distinct singular values of $A$ arranged 
in decreasing order and denote $\Delta_k:=\{i: \sigma_i=\mu_k\}, k=1,\dots, d.$ Let $\nu_k:={\rm card}(\Delta_k)$
be the multiplicity of $\mu_k, k=1,\dots, d.$ Denote 
\begin{equation*}
\begin{split}
P_k^{uv}:=\sum_{i\in \Delta_k} (u_i\otimes v_i),&\  P_k^{vu}:=\sum_{i\in \Delta_k} (v_i\otimes u_i),\\
P_k^{uu}:=\sum_{i\in \Delta_k} (u_i\otimes u_i),&\  P_k^{vv}:=\sum_{i\in \Delta_k} (v_i\otimes v_i).
\end{split}
\end{equation*}
It is straightforward to check that the following relationships hold:
\begin{equation}
\label{proj_uv}
(P_k^{uu})'=P_k^{uu},\  (P_k^{uu})^2= P_k^{uu},\ P_k^{vu}=(P_k^{uv})',\ P_k^{uv}P_k^{vu}=P_k^{uu}. 
\end{equation}
This implies, in particular, that the operators $P_k^{uu}, P_k^{vv}$ are orthogonal projectors 
(in the spaces ${\mathbb R}^m, {\mathbb R}^n,$ respectively). It is also easy to check 
that 
\begin{equation}
\label{proj_uv_1}
P_k^{uu}P_{k'}^{uu}=0,\ 
P_k^{vv}P_{k'}^{vv}=0,\ 
P_k^{vu}P_{k'}^{uv}=0,\ 
P_k^{uv}P_{k'}^{vu}=0,\ k\neq k'.
\end{equation}

The SVD of matrix $A$ 
can be rewritten as $A=\sum_{k=1}^d \mu_k P_k^{uv}$ and it can be shown that the operators $P_k^{uv}, k=1,\dots, d$
are uniquely defined. Let  
$$
B=\Lambda(A)=\Big(\begin{array}{cc}0&A\\A'&0\end{array}\Big)
= \sum_{k=1}^d \mu_k \Big(\begin{array}{cc}0&P_k^{uv}\\P_k^{vu}&0\end{array}\Big).
$$
For $k=1,\dots, d,$ denote 
$$
P_k := \frac{1}{2}\Big(\begin{array}{cc}P_k^{uu}&P_k^{uv}\\P_k^{vu}&P_k^{vv}\end{array}\Big),\ 
P_{-k}:=\frac{1}{2}\Big(\begin{array}{cc}P_k^{uu}&-P_k^{uv}\\-P_k^{vu}&P_k^{vv}\end{array}\Big),$$
and also 
$$
\mu_{-k}:=-\mu_k.
$$
Using relationships (\ref{proj_uv}), (\ref{proj_uv_1}), it is easy to show that $P_k P_{k'}=P_{k'}P_{k}=\mathbbm{1}(k=k')P_k$
for all $k,k', 1\leq |k|\leq d, 1\leq |k'|\leq d.$  Since the operators $P_k:{\mathbb R}^{m+n}\mapsto {\mathbb R}^{m+n}, 1\leq |k|\leq d$ 
are also symmetric, they are orthogonal projectors onto mutually orthogonal subspaces of ${\mathbb R}^{m+n}.$
Note that, by a simple algebra, 
$
B=\sum_{1\leq |k|\leq d} \mu_k P_k,  
$
implying that $\mu_k$ are distinct eigenvalues of $B$ and $P_k$ are the corresponding eigenprojectors.  
Note also that if $2\sum_{k=1}^d\nu_k<m+n,$ then zero is also an eigenvalue of $B$ (that will be denoted by $\mu_0$) of multiplicity 
$\nu_0:=n+m-2\sum_{k=1}^d\nu_k.$
Representation $A\mapsto B=\Lambda(A)=\Big(\begin{array}{cc}0&A\\A'&0\end{array}\Big)$ will play a crucial role 
in what follows since it allows to reduce the analysis of SVD for matrix $A$ to the spectral representation 
$B=\sum_{1\leq |k|\leq d} \mu_k P_k$. In particular, the operators $P_k^{uv}$ involved 
in the SVD $A=\sum_{k=1}^d \mu_k P_k^{uv}$ can be recovered from the eigenprojectors $P_k$ 
of matrix $B$ (hence, they are uniquely defined). Define also 
$\theta_i:= \frac{1}{\sqrt{2}}\Big(\begin{array}{cc}u_i\\v_i\end{array}\Big)$ and 
$\theta_{-i}:= \frac{1}{\sqrt{2}}\Big(\begin{array}{cc}u_i\\-v_i\end{array}\Big)$ for $i=1,\dots, r$
and let $\Delta_{-k}:=\{-i: i\in \Delta_k\}, k=1,\dots, d.$ Then, $\theta_i, 1\leq |i|\leq r$ are orthonormal 
eigenvectors of $B$ (not necessarily uniquely defined) corresponding to its non-zero eigenvalues $\sigma_1\geq \dots \geq \sigma_r >0>\sigma_{-r}\geq \dots \geq \sigma_{-1}$
with $\sigma_{-i}=-\sigma_i$ and
$$P_k=\sum_{i\in \Delta_k}(\theta_i\otimes \theta_i), 1\leq |k|\leq d.$$ 

It will be assumed in what follows that $A$ is perturbed by a random matrix $X\in\mathbb{R}^{m\times n}$
with i.i.d. entries $X_{ij}\sim\mathcal{N}(0,\tau^2)$ for some $\tau>0.$ Given the SVD of the perturbed 
matrix
$$\tilde{A}=A+X=\sum_{j=1}^{m\wedge n}\tilde{\sigma}_i(\tilde{u}_i\otimes \tilde{v}_i),$$ 
our main interest lies in estimating singular vectors $u_i$ and $v_i$ of the matrix 
$A$ in the case when its singular values $\sigma_i$ are distinct, or, more generally,
in estimating the operators $P_k^{uu}, P_k^{uv}, P_k^{vu}, P_k^{vv}.$ To this end, 
we will use the estimators 
\begin{equation*}
\begin{split}
\tilde P_k^{uu}:=\sum_{i\in \Delta_k}(\tilde u_i\otimes \tilde u_i),&\ 
\tilde P_k^{uv}:=\sum_{i\in \Delta_k}(\tilde u_i\otimes \tilde v_i),\\
\tilde P_k^{vu}:=\sum_{i\in \Delta_k}(\tilde v_i\otimes \tilde u_i),&\ 
\tilde P_k^{vv}:=\sum_{i\in \Delta_k}(\tilde v_i\otimes \tilde v_i), 
\end{split}
\end{equation*}
and our main goal will be to study the fluctuations of the bilinear forms 
of these random operators around the bilinear forms of operators $P_k^{uu}, P_k^{uv},$ $P_k^{vu}, P_k^{vv}.$
In the case when the singular values of $A$ are distinct, this would allow us 
to study the fluctuations of linear forms of singular vectors $\tilde u_i, \tilde v_i$ around 
the corresponding linear forms of $u_i, v_i$ which would provide a way to control 
the fluctuations of components of ``empirical'' singular vectors in a given basis 
around their true counterparts. Clearly, the problem can be and will be reduced 
to the analysis of spectral representation of a symmetric random matrix 
\begin{equation}
\label{BXconstruct}
\tilde{B}=\Lambda(\tilde A)=\Big(\begin{array}{cc}0&\tilde{A}\\\tilde{A}'&0\end{array}\Big)= B+\Gamma,\quad {\rm where}\ \Gamma=\Lambda(X)=\Big(\begin{array}{cc}0&X\\X'&0\end{array}\Big),
\end{equation}
that can be viewed as a random perturbation of the symmetric matrix $B.$ 
The spectral representation of this matrix can be written in the form 
$$
\tilde B= \sum_{1\leq |i|\leq (m\wedge n)}\tilde \sigma_i (\tilde \theta_i\otimes \tilde \theta_i),
$$
where 
$$
\tilde \sigma_{-i}=-\tilde \sigma_i,\ 
\tilde \theta_i:=\frac{1}{\sqrt{2}}\Big(\begin{array}{cc}\tilde u_i\\ \tilde v_i\end{array}\Big),\ 
\tilde \theta_{-i}:=\frac{1}{\sqrt{2}}\Big(\begin{array}{cc}\tilde u_i\\ -\tilde v_i\end{array}\Big),\ 
i=1,\dots, (m\wedge n).
$$
If the operator norm $\|\Gamma\|$ of the ``noise" matrix $\Gamma$   
is small enough comparing with the ``spectral gap" of the $k$-th eigenvalue $\mu_k$ of $B$ (for some $k=1,\dots, d$), then it 
is easy to see that $\tilde P_k:= \sum_{i\in \Delta_k}(\tilde \theta_i\otimes \tilde \theta_i)$ is the orthogonal 
projector on the direct sum of eigenspaces of $\tilde B$ corresponding to the ``cluster" $\{\tilde \sigma_i:i\in \Delta_k\}$ of its eigenvalues localized in a neighborhood 
of $\mu_k.$ Moreover, 
$\tilde P_k = \frac{1}{2}\Big(\begin{array}{cc}\tilde P_k^{uu}&\tilde P_k^{uv}\\\tilde P_k^{vu}&\tilde P_k^{vv}\end{array}\Big).$
Thus, it is enough to study the fluctuations of bilinear forms of random orthogonal projectors $\tilde P_k$ around 
the corresponding bilinear form of the spectral projectors $P_k$ to derive similar properties of operators $\tilde P_k^{uu}, \tilde P_k^{uv},
\tilde P_k^{vu}, \tilde P_k^{vv}.$


We will be interested in bounding the bilinear forms of operators $\tilde{P}_k-P_k$ for $k=1,\dots ,d.$ 
To this end, we will provide separate bounds on the random error $\tilde{P}_k-{\mathbb E}\tilde P_k$ and on the bias ${\mathbb E}\tilde P_k-P_k.$
 For $k=1,\dots, d,$ $\bar g_k$ denotes the distance from the eigenvalue $\mu_k$ to the rest of 
the spectrum of $A$ (the eigengap of $\mu_k$). More specifically, for $2\leq k\leq d-1$, $\bar{g}_k=\min(\mu_k-\mu_{k+1},\mu_{k-1}-\mu_k),$ $\bar{g}_1=\mu_1-\mu_2$ and $\bar{g}_d=\min(\mu_{d-1}-\mu_d,\mu_d).$ 


The main assumption in the results that follow is that ${\mathbb E}\|X\|<\frac{\bar g_k}{2}$ (more precisely, ${\mathbb E}\|X\|\leq (1-\gamma)\frac{\bar g_k}{2}$ for a positive $\gamma$).  In view of the concentration inequality of 
Lemma~\ref{spectraldevlem} in the next section, 
this essentially means that the operator norm of the random perturbation matrix $\|\Gamma\|=\|X\|$ is strictly smaller than 
one half of the spectral gap $\bar g_k$ of singular value $\mu_k.$ Since, again by  Lemma~\ref{spectraldevlem}, 
${\mathbb E}\|X\|\asymp \tau \sqrt{m\vee n},$ this assumption 
also means that $\bar g_k \gtrsim \tau\sqrt{m\vee n}$  (so, the spectral gap $\bar g_k$ is sufficiently large). 
Our goal is to prove that, under this assumption, the values of bilinear form $\langle \tilde P_k x,y\rangle$ of random spectral projector $\tilde P_k$ have tight concentration around their means (with the magnitude of deviations of 
the order $\sqrt{\frac{1}{m\vee n}}$). We will also show that the bias ${\mathbb E}\tilde P_k-P_k$ of the spectral 
projector $\tilde P_k$ is ``aligned" with the spectral projector $P_k$ (up to an error of the order $\sqrt{\frac{1}{m\vee n}}$
in the operator norm). More precisely, the following results hold.


\begin{thm}
\label{hatPrconthm}
Suppose that for some $\gamma\in(0,1),$
${\mathbb E}\|X\|\leq (1-\gamma)\frac{\bar g_k}{2}.$
There exists a constant $D_{\gamma}>0$ such that, for all $x,y\in\mathbb{R}^{m+n}$ and for all $t\geq 1,$ the following inequality holds with probability at least $1-e^{-t}:$
 \begin{equation}
 \big|\big<(\tilde{P}_k-\mathbb{E}\tilde{P}_k)x,y\big>\big|\leq D_{\gamma}\frac{\tau\sqrt{t}}{\bar{g}_k}
 \Big(\frac{\tau\sqrt{m\vee n}+\tau\sqrt{t}}{\bar{g}_k}+1\Big)\|x\|\|y\|.
 \end{equation}
\end{thm}

Assuming that $t\lesssim m\vee n$ and taking into account 
that $\tau \sqrt{m\vee n}\asymp {\mathbb E}\|X\|\leq \bar g_k,$ we easily get from the bound of Theorem \ref{hatPrconthm} that 
$$
\big|\big<(\tilde{P}_k-\mathbb{E}\tilde{P}_k)x,y\big>\big| \lesssim_{\gamma} 
\frac{\tau \sqrt{t}}{\bar g_k}\|x\|\|y\| \lesssim_{\gamma} \sqrt{\frac{t}{m\vee n}}\|x\|\|y\|, 
$$
so, the fluctuations of $\langle \tilde P_k x,y\rangle$ around its expectation are indeed of the order 
$\sqrt{\frac{1}{m\vee n}}.$


The next result shows that the bias ${\mathbb E}\tilde P_k-P_k$ of $\tilde P_k$ can be represented 
as a sum of a ``low rank part" $P_k({\mathbb E} \tilde P_k-P_k)P_k$ and a small 
remainder. 

\begin{thm}
\label{prdevthm}
The following bound holds with some constant $D>0:$
\begin{equation}
\label{biasA}
\Bigl\|{\mathbb E}\tilde P_k-P_k\Bigr\| \leq 
D\frac{\tau^2 (m\vee n)}{\bar g_k^2}.
\end{equation}
Moreover, suppose that for some $\gamma\in(0,1)$, 
 ${\mathbb E}\|X\|\leq (1-\gamma)\frac{\bar g_k}{2}.$
  Then, there exists a constant $C_{\gamma}>0$ such that
 \begin{equation}
 \label{biasB}
 \big\|\mathbb{E}\tilde{P}_k-P_k-P_k({\mathbb E} \tilde P_k-P_k)P_k\big\|\leq C_{\gamma}\frac{\nu_k \tau^2\sqrt{m\vee n}}{\bar{g}_k^2}.
 \end{equation}
\end{thm}

Since, under the assumption ${\mathbb E}\|X\|\leq (1-\gamma)\frac{\bar g_k}{2},$ we have $\bar g_k\gtrsim \tau\sqrt{m\vee n},$
bound (\ref{biasB}) implies that the following representation holds
$$
{\mathbb E}\tilde P_k -P_k = P_k({\mathbb E}\tilde P_k -P_k)P_k + T_k
$$
with the remainder $T_k$ satisfying the bound
$$
\|T_k\|\lesssim_{\gamma}\frac{\tau^2\sqrt{m\vee n}}{\bar{g}_k^2}\lesssim_{\gamma} \frac{\nu_k}{\sqrt{m\vee n}}.
$$

We will now consider a special case when $\mu_k$ has multiplicity $1$ ($\nu_k=1$). 
In this case, $\Delta_k=\{i_k\}$ for some $i_k\in\left\{1,\dots, (m\wedge n)\right\}$ and 
$P_k=\theta_{i_k}\otimes\theta_{i_k}.$ Let $\tilde{P}_k:=\tilde{\theta}_{i_k}\otimes\tilde{\theta}_{i_k}.$
Note that on the event $\|\Gamma\|=\|X\|<\frac{\bar g_k}{2}$ that is assumed to hold with a high probability,
the multiplicity of $\tilde \sigma_{i_k}$ is also $1$ (see the discussion in the next section after Lemma \ref{spectralmomentlem}). Note also that the unit eigenvectors $\theta_{i_k}, \tilde \theta_{i_k}$
are defined only up to their signs. Due to this, we will assume without loss of generality that $\langle \tilde \theta_{i_k}, \theta_{i_k}\rangle \geq 0.$
 
Since $P_k=\theta_{i_k}\otimes \theta_{i_k}$ is an operator of rank $1,$ we have 
$$
P_k({\mathbb E}\tilde P_k-P_k)P_k=b_k P_k,
$$
where 
$$
b_k:= \Bigl\langle ({\mathbb E}\tilde P_k-P_k)\theta_{i_k}, \theta_{i_k}\Bigr\rangle
={\mathbb E}\langle \tilde \theta_{i_k}, \theta_{i_k}\rangle^2-1. 
$$
Therefore,
\begin{equation*}
 \mathbb{E}\tilde{P}_k=(1+b_k)P_k+T_k
\end{equation*}
and $b_k$ turns out to be the main parameter characterizing 
the bias of $\tilde P_k.$ 
Clearly,
$
b_k\in [-1,0]
$
(note that $b_k=0$ is equivalent to $\tilde \theta_{i_k}=\theta_{i_k}$ a.s. and $b_k=-1$ is equivalent 
to $\tilde \theta_{i_k}\perp \theta_{i_k}$ a.s.). 
On the other hand, by bound (\ref{biasA}) of Theorem  \ref{prdevthm},
\begin{equation}
\label{bkbound}
|b_k|\leq \Bigl\|{\mathbb E}\tilde P_k-P_k\Bigr\|
\lesssim \frac{\tau^2 (m\vee n)}{\bar g_k^2}.
\end{equation}
In the next theorem, it will be assumed that the bias is not too large in the sense that $b_k$ is bounded away by 
a constant $\gamma>0$ from $-1.$ 


\begin{thm}
\label{thetadevthm}
Suppose that, for some $\gamma\in(0,1),$ ${\mathbb E}\|X\|\leq (1-\gamma)\frac{\bar{g}_k}{2}$  
and $1+b_k\geq \gamma.$ Then, for all 
$x\in\mathbb{R}^{m+n}$ and for all $t\geq 1$ with probability at least $1-e^{-t}$,
 \begin{equation*}
 \big|\big<\tilde{\theta}_{i_k}-\sqrt{1+b_k}\theta_{i_k},x\big>\big|\lesssim_{\gamma}
 \frac{\tau\sqrt{t}}{\bar{g}_k}
 \Big(\frac{\tau\sqrt{m\vee n}+\tau\sqrt{t}}{\bar{g}_k}+1\Big)\|x\|.
 \end{equation*}
\end{thm}

Assuming that $t\lesssim m\vee n,$ the bound of Theorem \ref{thetadevthm} implies that 
$$
\big|\big<\tilde{\theta}_{i_k}-\sqrt{1+b_k}\theta_{i_k},x\big>\big|\lesssim_{\gamma}
\frac{\tau\sqrt{t}}{\bar{g}_k}\|x\|\lesssim_{\gamma}\sqrt{\frac{t}{m\vee n}}\|x\|.
$$
Therefore, the fluctuations of $\langle \tilde \theta_{i_k},x\rangle$ around 
$\sqrt{1+b_k}\langle \theta_{i_k},x\rangle$ are of the order $\sqrt{\frac{1}{m\vee n}}.$

Recall that $\theta_{i_k}:= \frac{1}{\sqrt{2}}\Big(\begin{array}{cc}u_{i_k}\\v_{i_k}\end{array}\Big),$ 
where $u_{i_k}, v_{i_k}$ are left and right singular vectors of $A$ corresponding to its singular 
value $\mu_k.$ Theorem \ref{thetadevthm} easily implies the following corollary.

\begin{cor}
Under the conditions of Theorem \ref{thetadevthm},
with probability at least $1-\frac{1}{m+n}$,
\begin{equation*}
 \max\Big\{\big\|\tilde{u}_{i_k}-\sqrt{1+b_k}u_{i_k}\big\|_{\infty},\big\|\tilde{v}_{i_k}-\sqrt{1+b_k}v_{i_k}\big\|_{\infty}\Big\}\lesssim\sqrt{\frac{\log(m+n)}{m\vee n}}.
\end{equation*}
\end{cor}

For the proof, it is enough to take $t= 2\log (m+n),$ $x=e_{i}^{m+n}, i=1,\dots, (m+n)$ and to use the bound of 
Theorem \ref{thetadevthm} along with the union bound. 
Then recalling that $\theta_{i_k}=\frac{1}{\sqrt{2}}(u_{i_k}',v_{i_k}')',$ Theorem~\ref{thetadevthm} easily implies 
the claim.

Theorem \ref{thetadevthm} shows that the ``naive estimator" $\langle \tilde \theta_{i_k}, x\rangle$ of linear  
form $\langle \theta_{i_k},x\rangle$ could be improved by reducing its bias that, in principle,  could be done by its 
simple rescaling $\langle \tilde \theta_{i_k}, x\rangle\mapsto \langle (1+b_k)^{-1/2}\tilde \theta_{i_k}, x\rangle.$
Of course, the difficulty with this approach is related to the fact that the bias parameter $b_k$ is unknown.
We will outline below a simple approach based on repeated observations of matrix $A.$
More specifically, let $\tilde{A}^1=A+X^1$ and $\tilde{A}^2=A+X^2$ be two independent copies of $\tilde A$
and denote $\tilde B^1=\Lambda(\tilde A^1), \tilde B^2=\Lambda(\tilde A^2).$
Let $\tilde{\theta}_{i_k}^1$ and $\tilde{\theta}_{i_k}^2$
be the eigenvectors of $\tilde{B}^1$ and $\tilde{B}^2$ corresponding to their eigenvalues $\tilde \sigma^1_{i_k}, \tilde \sigma^2_{i_k}.$ The signs of $\tilde{\theta}_{i_k}^1$ and $\tilde{\theta}_{i_k}^2$ are chosen so that 
$\big<\tilde{\theta}_{i_k}^1,\tilde{\theta}_{i_k}^2\big>\geq 0$. Let
\begin{equation}
\label{breq}
 \tilde{b}_k:=\big<\tilde{\theta}_{i_k}^1,\tilde{\theta}_{i_k}^2\big>-1.
\end{equation}
Given $\gamma>0,$ define 
\begin{equation*}
 \hat{\theta}_{i_k}^{(\gamma)}:=\frac{\tilde{\theta}_{i_k}^1}{\sqrt{1+\tilde{b}_k}\vee \frac{\sqrt{\gamma}}{2}}.
\end{equation*}

\begin{cor}
\label{thetacor}
 Under the assumptions of Theorem~\ref{thetadevthm}, there exists a constant $D_{\gamma}>0$ such that for all $x\in {\mathbb R}^{m+n}$
 and all $t\geq 1$ with probability at least $1-e^{-t}$,
 \begin{equation}
 \label{bkbd}
 |\hat{b}_k-b_k|\leq D_{\gamma}\frac{\tau\sqrt{t}}{\bar{g}_k}
 \Big[\frac{\tau\sqrt{m\vee n}+\tau\sqrt{t}}{\bar{g}_k}+1\Big]
\end{equation}
 and 
 \begin{equation}
 \label{estbd}
 |\big<\hat{\theta}_{i_k}^{(\gamma)}-\theta_{i_k},x\big>|\leq D_{\gamma}\frac{\tau\sqrt{t}}{\bar{g}_k}
 \Big[\frac{\tau\sqrt{m\vee n}+\tau\sqrt{t}}{\bar{g}_k}+1\Big]\|x\|.
 \end{equation}
\end{cor}

Note that $\hat \theta_{i_k}^{(\gamma)}$ is not necessarily a unit vector. However, its linear form provides  
a better approximation of the linear forms of $\theta_{i_k}$ than in the case of vector $\tilde \theta_{i_k}^1$ that 
is properly normalized.   
Clearly, the result implies similar bounds for the singular vectors $\hat{u}_{i_k}^{(\gamma)}$ and $\hat{v}_{i_k}^{(\gamma)}$.

\section{Proofs of the main results}
\label{devhatPsec}

The proofs follow the approach of Koltchinskii and Lounici~\cite{koltchinskii2014asymptotics} who did a similar analysis in the problem of estimation of spectral projectors of sample covariance.  
We start with discussing several preliminary facts used in what follows. 
Lemma~\ref{spectraldevlem} and Lemma~\ref{spectralmomentlem} below provide moment bounds and a concentration 
inequality for $\|\Gamma\|=\|X\|.$
The bound on ${\mathbb E}\|X\|$ of Lemma~\ref{spectraldevlem} is available in many references (see, e.g., Vershynin~\cite{vershynin2010introduction}). The concentration bound for $\|X\|$ is a straightforward 
consequence of the Gaussian concentration inequality.
The moment bounds of Lemma~\ref{spectralmomentlem} can be easily proved by integrating out the tails of the exponential bound that follows from the concentration inequality of Lemma~\ref{spectraldevlem}.

\begin{lem}
\label{spectraldevlem}
 There exist absolute constants $c_0, c_1,c_2>0$ such that 
 \begin{equation*}
 c_0\tau\sqrt{m\vee n}\leq \mathbb{E}\|X\|\leq c_1\tau\sqrt{m\vee n}
\end{equation*}
and for all $t>0$,
\begin{equation*}
\mathbb{P}\big\{\bigl|\|X\|-\mathbb{E}\|X\|\bigr|\geq c_2\tau\sqrt{t}\big\}\leq e^{-t}.
\end{equation*}
\end{lem}

\begin{lem}
\label{spectralmomentlem}
 For all $p\geq 1$, it holds that
\begin{equation*}
 \mathbb{E}^{1/p}\|X\|^p\asymp \tau\sqrt{m\vee n}
\end{equation*}
\end{lem}

According to a well-known result that goes back to Weyl, for symmetric (or Hermitian) $N\times N$ matrices 
$C,D$ 
$$
\max_{1\leq j\leq N}\Bigl|\lambda_j^{\downarrow}(C)-\lambda_j^{\downarrow}(D)\Bigr|\leq \|C-D\|,
$$
where $\lambda^{\downarrow}(C), \lambda^{\downarrow}(D)$ denote the vectors consisting of the eigenvalues of 
matrices $C,D,$ respectively, arranged in a non-increasing order.  This immediately implies that, for all 
$k=1,\dots, d,$
$$
\max_{j\in \Delta_k}|\tilde \sigma_j-\mu_k| \leq \|\Gamma\|
$$
and 
$$
\min_{j\in \cup_{k'\neq k}\Delta_{k'}} |\tilde \sigma_j-\mu_k|\geq \bar g_k-\|\Gamma\|.
$$
Assuming that $\|\Gamma\|<\frac{\bar g_k}{2},$ we get that $\{\tilde \sigma_j:j\in \Delta_k\}\subset (\mu_k-\bar g_k/2, \mu_k+\bar g_k/2)$ and the rest of the eigenvalues of $\tilde B$ are outside of this interval. Moreover, if $\|\Gamma\|<\frac{\bar g_k}{4},$
then the cluster of eigenvalues $\{\tilde \sigma_j:j\in \Delta_k\}$ is localized inside a shorter interval $(\mu_k-\bar g_k/4, \mu_k+\bar g_k/4)$ of radius $\bar g_k/4$ and its distance from the rest of the spectrum of $\tilde B$ is $>\frac{3}{4}\bar g_k.$
These simple considerations allow us to view the projection operator $\tilde P_k=\sum_{j\in \Delta_k} (\tilde \theta_j\otimes \tilde \theta_j)$ as a projector on the direct sum of eigenspaces of $\tilde B$ corresponding to its eigenvalues located in 
a ``small" neighborhood of the eigenvalue $\mu_k$ of $B,$ which makes $\tilde P_k$ a natural estimator of $P_k.$

Define operators $C_k$ as follows:
\begin{equation*}
 C_k=\sum_{s\neq k}\frac{1}{\mu_s-\mu_k}P_s.
\end{equation*}
In the case when $2\sum_{k=1}^d \nu_k<m+n$ and, hence, $\mu_0=0$ is also an eigenvalue of $B,$
it will be assumed that the above sum includes $s=0$ with $P_0$ being  the corresponding spectral 
projector.


The next simple lemma can be found, for instance, in Koltchinskii and Lounici~\cite{koltchinskii2014asymptotics}. 
Its proof is based on a standard perturbation analysis utilizing Riesz formula for spectral projectors.

\begin{lem}
\label{conlem}
The following bound holds:
\begin{eqnarray*}
 \|\tilde{P}_k-P_k\|\leq 4\frac{\|\Gamma\|}{\bar{g}_k}.
\end{eqnarray*}
Moreover,  
\begin{eqnarray*} 
 \tilde{P}_k-P_k=L_k(\Gamma)+S_k(\Gamma),
\end{eqnarray*}
where $L_k(\Gamma):=C_k\Gamma P_k+P_k\Gamma C_k$ and 
\begin{eqnarray*}
\|S_k(\Gamma)\|\leq 14\left(\frac{\|\Gamma\|}{\bar{g}_k}\right)^2.
\end{eqnarray*}
\end{lem}


\begin{proof}[\bf{Proof of Theorem~\ref{hatPrconthm}}]
Since $\mathbb{E}L_k(\Gamma)=0,$
it is easy to check that 
\begin{equation}
\label{repres}
\tilde{P}_k-\mathbb{E}\tilde{P}_k=L_k(\Gamma)+S_k(\Gamma)-\mathbb{E}S_k(\Gamma)=:
L_k(\Gamma)+R_k(\Gamma).
\end{equation}
We will first provide a bound on the bilinear form of the remainder $\big<R_k(\Gamma)x,y\big>.$ 
Note that 
$$
\left<R_k(\Gamma)x,y\right>=\left<S_k(\Gamma)x,y\right>-\left<\mathbb{E}S_k(\Gamma)x,y\right>
$$ 
is a function of the random matrix $X\in\mathbb{R}^{m\times n}$ since $\Gamma=\Lambda(X)$
(see (\ref{BXconstruct})).
When we need to emphasize this dependence, we will write $\Gamma_X$ instead of $\Gamma.$
With some abuse of notation, we will view $X$ as a point in ${\mathbb R}^{m\times n}$ rather than 
a random variable.


Let $0<\gamma<1$ and define a function $h_{x,y,\delta}(\cdot):\mathbb{R}^{m\times n}\to\mathbb{R}$
as follows:
$$
h_{x,y,\delta}(X):=\left<S_k(\Gamma_X)x,y\right>\phi\biggl(\frac{\|\Gamma_X\|}{\delta}\biggr),
$$
where $\phi$ is a Lipschitz function with constant $\frac{1}{\gamma}$ on $\mathbb{R}_{+}$ and $0\leq\phi(s)\leq 1$. More precisely, assume that $\phi(s)=1, s\leq 1,$
$\phi(s)=0, s\geq (1+\gamma)$ and $\phi$ is linear in between.
We will prove that the function $X\mapsto h_{x,y,\delta}(X)$ satisfy the Lipschitz condition. Note that 
$$
\left|\left<\left(S_k(\Gamma_{X_1})-S_k(\Gamma_{X_2})\right)x,y\right>\right|\leq \|S_k(\Gamma_{X_1})-S_k(\Gamma_{X_2})\|\|x\|\|y\|.
$$
To control the norm $\|S_k(\Gamma_{X_1})-S_k(\Gamma_{X_2})\|$, we need to apply Lemma 4 from  \cite{koltchinskii2014asymptotics}. It is stated below without the proof. 


\begin{lem}
\label{sconlem}
Let $\gamma\in(0,1)$ and suppose that $\delta\leq \frac{1-\gamma}{1+\gamma}\frac{\bar{g}_k}{2}.$ 
There exists a constant $C_{\gamma}>0$ such that, for all symmetric $\Gamma_1, \Gamma_2\in\mathbb{R}^{(m+n)\times (m+n)}$ 
satisfying the conditions $\|\Gamma_1\|\leq (1+\gamma)\delta$ 
and $\|\Gamma_2\|\leq (1+\gamma)\delta,$
 \begin{equation*}
\|S_k(\Gamma_1)-S_k(\Gamma_2)\|\leq C_{\gamma}\frac{\delta}{\bar{g}_k^2}\|\Gamma_1-\Gamma_2\|.
 \end{equation*}
\end{lem}

We now derive the Lipschitz condition for the function $X\mapsto h_{x,y,\delta}(X).$

\begin{lem}
 \label{huvliplem}
Under the assumption that $\delta\leq\frac{1-\gamma}{1+\gamma}\frac{\bar{g}_k}{2}$, there exists a constant $C_{\gamma}>0$,
 \begin{equation}
 \label{huvliplemineq1}
 \left|h_{x,y,\delta}(X_1)-h_{x,y,\delta}(X_2)\right|\leq C_{\gamma}\frac{\delta \|X_1-X_2\|_2}{\bar{g}_k^2}\|x\|\|y\|.
 \end{equation}
\end{lem}

\begin{proof}
Suppose first that $\max(\|\Gamma_{X_1}\|,\|\Gamma_{X_2}\|)\leq (1+\gamma)\delta.$
Using Lemma \ref{sconlem} and Lipschitz properties of function $\phi,$ we get  
\begin{equation*}
\begin{split}
|h_{x,y,\delta}(X_1)-&h_{x,y,\delta}(X_2)|=\left|\big<S_k(\Gamma_{X_1})x,y\big>\phi\biggl(\frac{\|\Gamma_{X_1}\|}{\delta}\biggr)-\big<S_k(\Gamma_{X_2})x,y\big>\phi\biggl(\frac{\|\Gamma_{X_2}\|}{\delta}\biggr)\right|\\
 \leq&\|S_k(\Gamma_{X_1})-S_k(\Gamma_{X_2})\|\|x\|\|y\|\phi\biggl(\frac{\|\Gamma_{X_1}\|}{\delta}\biggr)\\
+&\|S_k(\Gamma_{X_2})\|\left|\phi\biggl(\frac{\|\Gamma_{X_1}\|}{\delta}\biggr)-\phi\biggl(\frac{\|\Gamma_{X_2}\|}{\delta}\biggr)
\right|\|x\|\|y\|\\
 \leq& C_{\gamma}\frac{\delta\|\Gamma_{X_1}-\Gamma_{X_2}\|}{\bar{g}_k^2}\|x\|\|y\|
 + \frac{14(1+\gamma)^2\delta^2}{\bar{g}_k^2}\frac{\|\Gamma_{X_1}-\Gamma_{X_2}\|}{\gamma\delta}\|x\|\|y\|\\
 \lesssim_{\gamma}&
\frac{\delta\|\Gamma_{X_1}-\Gamma_{X_2}\|}{\bar{g}_k^2}\|x\|\|y\|\lesssim_{\gamma} \frac{\delta \|X_1-X_2\|_2}{\bar{g}_k^2}\|x\|\|y\|.
\end{split}
\end{equation*}
In the case when $\min(\|\Gamma_{X_1}\|,\|\Gamma_{X_2}\|)\geq (1+\gamma)\delta$, 
we have $h_{x,y,\delta}(X_1)=h_{x,y,\delta}(X_2)=0,$ and (\ref{huvliplemineq1}) trivially holds.
Finally, in the case when $\|\Gamma_{X_1}\|\leq(1+\gamma)\delta\leq \|\Gamma_{X_2}\|$, we have
\begin{equation*}
 \begin{split}
|h_{x,y,\delta}(X_1)-&h_{x,y,\delta}(X_2)|=\left|\big<S_k(\Gamma_{X_1})x,y\big>\phi\biggl(\frac{\|\Gamma_{X_1}\|}{\delta}\biggr)\right|\\
 =&\left|\big<S_k(\Gamma_{X_1})x,y\big>\phi\biggl(\frac{\|\Gamma_{X_1}\|}{\delta}\biggr)-\big<S_k(\Gamma_{X_1})x,y\big>
 \phi\biggl(\frac{\|\Gamma_{X_2}\|}{\delta}\biggr)\right|\\
 \leq&\|S_k(\Gamma_{X_1})\|\left|\phi\biggl(\frac{\|\Gamma_{X_1}\|}{\delta}\biggr)-\phi\biggl(\frac{\|\Gamma_{X_2}\|}{\delta}\biggr)\right|
 \|x\|\|y\|\\
 \leq&14\left(\frac{(1+\gamma)\delta}{\bar{g}_k}\right)^2\frac{\|\Gamma_{X_1}-\Gamma_{X_2}\|}{\gamma\delta}\|x\|\|y\|\\
 \lesssim_{\gamma}&\frac{\delta \|X_1-X_2\|_2}{\bar{g}_k^2}\|x\|\|y\|.
\end{split}
\end{equation*}
The case $\|\Gamma_{X_2}\|\leq (1+\gamma)\delta\leq \|\Gamma_{X_1}\|$ is similar. 
\end{proof}


Our next step is to apply the following concentration bound that easily follows from the Gaussian isoperimetric inequality.

\begin{lem}\label{medconlem}
 Let $f:{\mathbb R}^{m\times n}\mapsto {\mathbb R}$ be a function satisfying the following 
 Lipschitz condition with some constant $L>0:$
 $$
 |f(A_1)-f(A_2)|\leq L\|A_1-A_2\|_2, A_1,A_2\in {\mathbb R}^{m\times n}
 $$ 
 Suppose $X$ is a random $m\times n$ matrix with i.i.d. entries $X_{ij}\sim\mathcal{N}(0,\tau^2).$ 
 Let $M$ be a real number such that
 \begin{equation*}
 \mathbb{P}\big\{f(X)\geq M\big\}\geq \frac{1}{4}\text{ and } \mathbb{P}\big\{f(X)\leq M\big\}\geq \frac{1}{4}.
\end{equation*}
Then there exists some constant $D_1>0$ such that for all $t\geq 1$,
\begin{equation*}
 \mathbb{P}\Big\{\big|f(X)-M\big|\geq D_1L\tau\sqrt{t}\Big\}\leq e^{-t}.
\end{equation*}
\end{lem}

The next lemma is the main ingredient in the proof of Theorem \ref{hatPrconthm}.
It provides a Bernstein type bound on the bilinear form $\left<R_k(\Gamma)x,y\right>$
of the remainder $R_k$ in the representation (\ref{repres}).  

\begin{lem}
\label{rrdevthm}
Suppose that, for some $\gamma\in(0,1),$ ${\mathbb E}\|\Gamma\|\leq (1-\gamma)\frac{\bar g_k}{2}.$ 
Then, there exists a constant $D_{\gamma}>0$ such that 
 for all $x,y\in\mathbb{R}^{m+n}$ and all $t\geq \log(4)$,
 the following inequality holds with probability at least $1-e^{-t}$ 
\begin{equation*}
 \left|\left<R_k(\Gamma)x,y\right>\right|\leq 
 D_{\gamma}\frac{\tau\sqrt{t}}{\bar{g}_k}\biggl(\frac{\tau\sqrt{m\vee n}+\tau\sqrt{t}}{\bar{g}_k}\biggr)\|x\|\|y\|.
 \end{equation*}
\end{lem}

\begin{proof} Define $\delta_{n,m}(t):={\mathbb E}\|\Gamma\|+c_2\tau\sqrt{t}.$
By the second bound of Lemma \ref{spectraldevlem}, with a proper choice of constant $c_2>0,$
$\mathbb{P}\{\|\Gamma\|\geq \delta_{n,m}(t)\}\leq e^{-t}.$
We first consider the case when $c_2\tau\sqrt{t}\leq \frac{\gamma}{2}\frac{\bar{g}_k}{2},$ 
which implies that  
$$
\delta_{n,m}(t)\leq (1-\gamma/2)\frac{\bar g_k}{2}= \frac{1-\gamma'}{1+\gamma'}\frac{\bar{g}_k}{2}
$$
for some $\gamma'\in (0,1)$ depending only on $\gamma.$
Therefore, it enables us to use Lemma~\ref{huvliplem} with $\delta:=\delta_{n,m}(t).$ Recall that $h_{x,y,\delta}(X)=\big<S_k(\Gamma)x,y\big>\phi\biggl(\frac{\|\Gamma\|}{\delta}\biggr)$
and let $M:=\text{Med}\big(\big<S_k(\Gamma)x,y\big>\big)$. Observe that, for $t\geq \log(4),$
\begin{equation*}
\begin{split}
\mathbb{P}&\{h_{x,y,\delta}(X)\geq M\}\geq \mathbb{P}\{h_{x,y,\delta}(X)\geq M, \|\Gamma\|\leq \delta_{n,m}(t)\}\\
\geq&\mathbb{P}\{\big<S_k(\Gamma)x,y\big>\geq M\big\}-\mathbb{P}\{\|\Gamma\|> \delta_{n,m}(t)\}\geq 
\frac{1}{2}-e^{-t}\geq\frac{1}{4}
\end{split}
\end{equation*}
and, similarly. $\mathbb{P}(h_{x,y,\delta}(X)\leq M)\geq\frac{1}{4}.$ Therefore, by applying lemmas~\ref{huvliplem},\ref{medconlem}, we conclude that with probability at least $1-e^{-t}$,
\begin{equation*}
 \big|h_{x,y,\delta}(X)-M\big|\lesssim_{\gamma}\frac{\delta_{n,m}(t)\tau\sqrt{t}}{\bar{g}_k^2}\|x\|\|y\|
\end{equation*}
Since, by the first bound of Lemma \ref{spectraldevlem},
$\delta_{n,m}(t)\lesssim \tau (\sqrt{m\vee n}+\sqrt{t}),$ we get that with the same probability 
\begin{equation*}
 \big|h_{x,y,\delta}(X)-M\big|\lesssim_{\gamma}\frac{\tau\sqrt{t}}{\bar{g}_k}\frac{\tau \sqrt{m\vee n}+\tau \sqrt{t}}{\bar g_k}\|x\|\|y\|.
\end{equation*}
Moreover, on the event $\{\|\Gamma\|\leq \delta_{n,m}(t)\}$ that holds with probability at least $1-e^{-t},$ 
$h_{x,y,\delta}(X)=\big<S_k(\Gamma)x,y\big>.$
Therefore, the following inequality holds with probability at least $1-2e^{-t}:$
\begin{equation}
\label{rrdevthmineq1}
 \big|\big<S_k(\Gamma)x,y\big>-M\big|\lesssim_{\gamma}\frac{\tau\sqrt{t}}{\bar{g}_k}\frac{\tau \sqrt{m\vee n}+\tau \sqrt{t}}{\bar g_k}\|x\|\|y\|.
 \end{equation}
We still need to prove a similar inequality in the case $c_2\tau\sqrt{t}\geq \frac{\gamma}{2}\frac{\bar{g}_k}{2}.$ 
In this case,
$$
{\mathbb E}\|\Gamma\|\leq (1-\gamma)\frac{\bar g_k}{2}\leq \frac{2c_2(1-\gamma)}{\gamma}\tau\sqrt{t},
$$
implying that $\delta_{n,m}(t)\lesssim_{\gamma} \tau\sqrt{t}.$
It follows from Lemma~\ref{conlem} that
\begin{equation*}
 \big|\big<S_k(\Gamma)x,y\big>\big|\leq \|S_k(\Gamma)\|\|x\|\|y\|\lesssim\frac{\|\Gamma\|^2}{\bar{g}_k^2}\|x\|\|y\|
\end{equation*}
This implies that with probability at least $1-e^{-t}$,
\begin{equation*}
 \big|\big<S_k(\Gamma)x,y\big>\big|\lesssim\frac{\delta_{n,m}^2(t)}{\bar{g}_k^2}\|x\|\|y\|
 \lesssim_{\gamma}  \frac{\tau^2 t}{\bar g_k^2}\|x\|\|y\|. 
\end{equation*}
Since $t\geq \log(4)$ and $e^{-t}\leq 1/4,$ we can bound the median $M$ of $\big<S_k(\Gamma)x,y\big>$
as follows:
\begin{equation*}
M \lesssim_{\gamma}  \frac{\tau^2 t}{\bar g_k^2}\|x\|\|y\|,
\end{equation*}
which immediately implies that bound (\ref{rrdevthmineq1}) holds under assumption $c_2\tau\sqrt{t}\geq \frac{\gamma}{2}\frac{\bar{g}_k}{2}$ as well.
By integrating out the tails of exponential bound (\ref{rrdevthmineq1}), we obtain that
\begin{equation*}
 \big|\mathbb{E}\big<S_k(\Gamma)x,y\big>-M\big|\leq \mathbb{E}\big|\big<S_k(\Gamma)x,y\big>-M\big|\lesssim_{\gamma} \frac{\tau^2\sqrt{m\vee n}}{\bar{g}_k^2}\|x\|\|y\|,
\end{equation*}
which allows us to replace the median by the mean in concentration inequality (\ref{rrdevthmineq1}).
To complete the proof, it remains to rewrite the probability bound $1-2e^{-t}$ as $1-e^{-t}$ by 
adjusting the value of the constant $D_{\gamma}.$ 
\end{proof}

Recalling that $\tilde{P}_k-\mathbb{E}\tilde{P}_k=L_k(\Gamma)+R_k(\Gamma),$ it remains to study the concentration of $\big<L_k(\Gamma)x,y\big>$.

\begin{lem}
 \label{Lrdevlem}
 For all $x,y\in\mathbb{R}^{m+n}$ and $t>0$,
 \begin{equation*}
  \mathbb{P}\left(\big|\big<L_k(\Gamma)x,y\big>\big|\geq 4\frac{\tau\|x\|\|y\|\sqrt{t}}{\bar{g}_k}\right)\leq e^{-t}.
 \end{equation*}
\end{lem}

\begin{proof}
Recall that $L_k(\Gamma)=P_k \Gamma C_k+C_k\Gamma P_k$ implying that 
$$
\langle L_k(\Gamma)x,y\rangle = \langle \Gamma P_k x, C_k y\rangle + \langle \Gamma C_k x, P_k y\rangle.
$$
If 
$x=\Big(\begin{array}{cc}x_1\\x_2\end{array}\Big), y=\Big(\begin{array}{cc}y_1\\y_2\end{array}\Big),$
where $x_1,y_1\in {\mathbb R}^m, x_2,y_2\in {\mathbb R}^n,$ then it is easy to check that 
$$
\langle \Gamma x,y\rangle = \langle Xx_2,y_1\rangle + \langle Xy_2,x_1\rangle.
$$
Clearly, the random variable $\langle \Gamma x,y\rangle$ is normal with mean zero and variance 
\begin{eqnarray*}
{\mathbb E}\langle \Gamma x,y\rangle^2\leq 
2\Bigl[ {\mathbb E}\langle Xx_2,y_1\rangle^2 + {\mathbb E}\langle Xy_2,x_1\rangle^2\Bigr].
\end{eqnarray*}
Since $X$ is an $m\times n$ matrix with i.i.d. ${\mathcal N}(0,\tau^2)$ entries, we easily get that 
$$
{\mathbb E}\langle X x_2,y_1\rangle ^2 = {\mathbb E}\langle X, y_1\otimes x_2\rangle^2
=\tau^2 \|y_1\otimes x_2\|_2^2 = \tau^2 \|x_2\|^2 \|y_1\|^2
$$ 
and, similarly, 
$$
{\mathbb E}\langle Xy_2,x_1\rangle^2=\tau^2 \|x_1\|^2 \|y_2\|^2.
$$
Therefore, 
\begin{equation*}
\begin{split}
{\mathbb E}\langle \Gamma x,y\rangle^2\leq& 
2\tau^2 \Bigl[ \|x_2\|^2 \|y_1\|^2+\|x_1\|^2 \|y_2\|^2\Bigr]\\
\leq& 2\tau^2 \Bigl[(\|x_1\|^2+\|x_2\|^2)(\|y_1\|^2+\|y_2\|^2)\Bigr]
=2\tau^2 \|x\|^2 \|y\|^2.
\end{split}
\end{equation*}
As a consequence, the random variable $\langle L_k(\Gamma)x,y\rangle $ is also normal with mean zero 
and its variance is bounded from above as follows:
\begin{equation*}
\begin{split}
{\mathbb E}\langle L_k(\Gamma)x,y\rangle^2 \leq& 
2\Bigl[{\mathbb E}\langle \Gamma P_k x, C_k y\rangle^2 + {\mathbb E}\langle \Gamma C_k x, P_k y\rangle^2\Bigr]\\
\leq&
4\tau^2 \Bigl[\|P_kx\|^2 \|C_k y\|^2+ \|C_k x\|^2 \|P_k y\|^2\Bigr].
\end{split}
\end{equation*}
Since $\|P_k\|\leq 1$ and $\|C_k\|\leq \frac{1}{\bar g_k},$ we get that
$$
{\mathbb E}\langle L_k(\Gamma)x,y\rangle^2\leq \frac{8\tau^2}{\bar g_k^2}\|x\|^2 \|y\|^2.
$$
The bound of the lemma easily follows from standard tail bounds for normal random variables. 
\end{proof}

The upper bound on $|\big<(\tilde{P}_k-\mathbb{E}\tilde{P}_k)x,y\big>|$ claimed in Theorem~\ref{hatPrconthm} follows by combining Lemma~\ref{rrdevthm} and Lemma~\ref{Lrdevlem}.
\end{proof}


\begin{proof}[\bf{Proof of Theorem~\ref{prdevthm}}]
Note that, since $\tilde P_k-P_k=L_k(\Gamma)+S_k(\Gamma)$ and $\mathbb{E}L_k(\Gamma)=0$, we have
$$
{\mathbb E}\tilde P_k-P_k= {\mathbb E}S_k(\Gamma).
$$
It follows from the bound on $\|S_k(\Gamma)\|$ of Lemma \ref{conlem} that  
\begin{equation}
\label{bias-BB}
\Bigl\|{\mathbb E}\tilde P_k-P_k\Bigr\| \leq {\mathbb E}\|S_k(\Gamma)\|\leq 
14 \frac{{\mathbb E}\|\Gamma\|^2}{\bar g_k^2}
\end{equation}
and the bound of Lemma \ref{spectralmomentlem} implies that 
$$
\Bigl\|{\mathbb E}\tilde P_k-P_k\Bigr\|\lesssim \frac{\tau^2 (m\vee n)}{\bar g_k^2},
$$
which proves (\ref{biasA}).

Let 
$$\delta_{n,m}:=\mathbb{E}\|\Gamma\|+c_2\tau\sqrt{\log(m+n)}.$$
It follows from Lemma  \ref{spectraldevlem} that, with a proper choice
of constant $c_2>0,$ 
$$\mathbb{P}\left(\|\Gamma\|\geq\delta_{n,m}\right)\leq\frac{1}{m+n}.$$
In the case when $c_2\tau\sqrt{\log(m+n)} >\frac{\gamma}{2}\frac{\bar g_k}{2},$
the proof of bound (\ref{biasB}) is trivial. Indeed, in this case 
$$
\Big\|{\mathbb E}\tilde P_k-P_k\Big\|\ \leq {\mathbb E}\|\tilde P_k\|+\|P_k\|\leq 2 \lesssim_{\gamma}
\frac{\tau^2 \log(m+n)}{\bar g_k^2} \lesssim \frac{\nu_k \tau^2 \sqrt{m\vee n}}{\bar g_k^2}. 
$$
Since $\Big\|P_k({\mathbb E}\tilde P_k-P_k)P_k\Big\|\leq \Big\|{\mathbb E}\tilde P_k-P_k\Big\|,$
bound (\ref{biasB}) of the theorem follows when $c_2\tau\sqrt{\log(m+n)} >\frac{\gamma}{2}\frac{\bar g_k}{2}.$

In the rest of the proof, it will be assumed that $c_2\tau\sqrt{\log(m+n)} \leq \frac{\gamma}{2}\frac{\bar g_k}{2}$ which, together with the 
condition ${\mathbb E}\|\Gamma\|={\mathbb E}\|X\|\leq (1-\gamma)\frac{\bar g_k}{2},$ implies 
that $\delta_{n,m}\leq (1-\gamma/2)\frac{\bar g_k}{2}.$ On the other hand, $\delta_{n,m}\lesssim \tau\sqrt{m\vee n}.$
The following decomposition of the bias $\mathbb{E}\tilde{P}_k-P_k$ is obvious:
\begin{equation}
\label{rep-rep}
\begin{split}
  \mathbb{E}\tilde{P}_k-P_k&=\mathbb{E}S_k(\Gamma)=\mathbb{E}P_kS_k(\Gamma)P_k\\
  +&\mathbb{E}\left(P_k^{\perp}S_k(\Gamma)P_k+P_kS_k(\Gamma)P_k^{\perp}+P_k^{\perp}S_k(\Gamma)P_k^{\perp}\right)\mathbbm{1}(\|\Gamma\|\leq\delta_{n,m})\\
  +&\mathbb{E}\left(P_k^{\perp}S_k(\Gamma)P_k+P_kS_k(\Gamma)P_k^{\perp}+P_k^{\perp}S_k(\Gamma)P_k^{\perp}\right)\mathbbm{1}(\|\Gamma\|>\delta_{n,m})
 \end{split}
\end{equation}
We start with bounding the part of the 
expectation in the right hand side of (\ref{rep-rep}) that corresponds to the  event $\{\|\Gamma\|\leq \delta_{n,m}\}$
on which we also have $\|\Gamma\|<\frac{\bar g_k}{2}.$
Under this assumption, 
the eigenvalues $\mu_k$ of $B$ and $\sigma_j(\tilde{B}),j\in\Delta_k$ of $\tilde{B}$ are inside
the circle $\gamma_k$ in $\mathbb{C}$ with center $\mu_k$ and radius $\frac{\bar{g}_k}{2}.$ 
The rest of the eigenvalues of $B,\tilde{B}$ are outside of $\gamma_k.$
According to the Riesz formula for spectral projectors,
\begin{equation*}
 \tilde{P}_k=-\frac{1}{2\pi i}\oint_{\gamma_k}R_{\tilde{B}}(\eta)d\eta,
\end{equation*}
where $R_{T}(\eta)=(T-\eta I)^{-1}, \eta\in {\mathbb C}\setminus \sigma(T)$ denotes the resolvent of operator $T$
($\sigma(T)$ being its spectrum). It is also assumed that the contour $\gamma_k$ has a counterclockwise orientation.   
Note that the resolvents will be viewed as operators from ${\mathbb C}^{m+n}$ into itself.
The following power series expansion is standard:
\begin{equation*}
\begin{split}
 R_{\tilde{B}}(\eta)=&R_{B+\Gamma}(\eta)=(B+\Gamma-\eta I)^{-1}\\
 =&[(B-\eta I)(I+(B-\eta I)^{-1}\Gamma)]^{-1}\\
 =&(I+R_{B}(\eta)\Gamma)^{-1}R_{B}(\eta)=\sum\limits_{r\geq 0}(-1)^r[R_{B}(\eta)\Gamma]^rR_{B}(\eta),
 \end{split}
\end{equation*}
where the series in the last line converges because $\|R_{B}(\eta)\Gamma\|\leq \|R_{B}(\eta)\|\|\Gamma\|<
\frac{2}{\bar{g}_k}\frac{\bar{g}_k}{2}=1$.
The inequality $\|R_{B}(\eta)\|\leq \frac{2}{\bar{g}_k}$ holds for all $\eta\in \gamma_k.$ 
One can easily verify that
\begin{equation*}
\begin{split}
 P_k=&-\frac{1}{2\pi i}\oint_{\gamma_k}R_B(\eta)d\eta,\\
 L_k(\Gamma)=&\frac{1}{2\pi i}\oint_{\gamma_k}R_B(\eta)\Gamma R_B(\eta)d\eta,\\
 S_k(\Gamma)=&-\frac{1}{2\pi i}\oint_{\gamma_k}\sum\limits_{r\geq 2}(-1)^r[R_B(\eta)\Gamma]^rR_{B}(\eta)d\eta.
 \end{split}
\end{equation*}

The following spectral representation of the resolvent will be used 
$$
R_B(\eta)=\sum\limits_{s}\frac{1}{\mu_s-\eta}P_s,
$$ 
where the sum in the right hand side includes $s=0$ in the case when $\mu_0=0$ is an eigenvalue 
of $B$ (equivalently, in the case when $2\sum_{k=1}^d\nu_k< m+n$). 
Define
\begin{equation*}
 \tilde{R}_{B}(\eta):=R_{B}(\eta)-\frac{1}{\mu_k-\eta}P_k=\sum\limits_{s\neq k}\frac{1}{\mu_s-\eta}P_s.
\end{equation*}
Then, for $r\geq 2,$
\begin{equation*}
 \begin{split}
  P_k^{\perp}&[R_{B}(\eta)\Gamma]^rR_{B}(\eta)P_k=\frac{1}{\mu_k-\eta}P_k^{\perp}[R_{B}(\eta)\Gamma]^rP_k\\
  =&\frac{1}{(\mu_k-\eta)^2}\sum\limits_{s=2}^r(\tilde{R}_{B}(\eta)\Gamma)^{s-1}P_k\Gamma(R_{B}(\eta)\Gamma)^{r-s}P_k+\frac{1}{\mu_k-\eta}(\tilde{R}_{B}(\eta)\Gamma)^rP_k.
 \end{split}
\end{equation*}
The above representation easily follows from the following simple observation: 
let $a:=\frac{P_k}{\mu_k-\eta}\Gamma$ and $b:=\tilde{R}_{B}(\eta)\Gamma.$
Then 
\begin{equation*}
 \begin{split}
  (a+b)^r=&a(a+b)^{r-1}+b(a+b)^{r-1}\\
 =&a(a+b)^{r-1}+ba(a+b)^{r-2}+b^2(a+b)^{r-2}\\
=&a(a+b)^{r-1}+ba(a+b)^{r-2}+b^2a(a+b)^{r-3}+b^3(a+b)^{r-3}\\
=&\ldots=\sum_{s=1}^rb^{s-1}a(a+b)^{r-s}+b^r.
 \end{split}
\end{equation*}
As a result,
\begin{equation}
\label{rep-rep-main}
\begin{split}
 P_k^{\perp}S_k(\Gamma)&P_k=-\sum\limits_{r\geq 2}(-1)^r\frac{1}{2\pi i}\oint_{\gamma_k}\Bigg[\frac{1}{(\mu_k-\eta)^2}\sum\limits_{s=2}^r(\tilde{R}_{B}(\eta)\Gamma)^{s-1}P_k\Gamma(R_{B}(\eta)\Gamma)^{r-s}P_k\\ 
 +&\frac{1}{\mu_k-\eta}(\tilde{R}_{B}(\eta)\Gamma)^rP_k\Bigg]d\eta\\
 \end{split}
\end{equation}
Let $P_k=\sum\limits_{l\in\Delta_k}\theta_l\otimes\theta_l,$ 
where $\{\theta_l,l\in\Delta_k\}$ are orthonormal eigenvectors corresponding to
the eigenvalue $\mu_k.$ Therefore, for any $y\in\mathbb{R}^{m+n}$,
\begin{equation}
\label{r-s}
 \begin{split}
  (\tilde{R}_{B}(\eta)\Gamma)&^{s-1}P_k\Gamma(R_{B}(\eta)\Gamma)^{r-s}P_k y
=\sum\limits_{l\in\Delta_k}(\tilde{R}_{B}(\eta)\Gamma)^{s-1}\theta_l\otimes\theta_l\Gamma(R_{B}(\eta)\Gamma)^{r-s}
P_ky\\
  =&\sum\limits_{l\in\Delta_k}\left<\Gamma(R_{B}(\eta)\Gamma)^{r-s}P_k y,\theta_l\right>(\tilde{R}_{B}(\eta)\Gamma)^{s-2}\tilde{R}_{B}(\eta)\Gamma\theta_l
 \end{split}
\end{equation}
Since $|\left<\Gamma(R_{B}(\eta)\Gamma)^{r-s}P_k y,\theta_l\right>|\leq 
\|\Gamma\|^{r-s+1} \|R_{B}(\eta)\|^{r-s}\|y\|$, we get
\begin{equation*}
 \mathbb{E}|\left<\Gamma(R_{B}(\eta)\Gamma)^{r-s}P_k y,\theta_l\right>|^2
\mathbbm{1}(\|\Gamma\|\leq\delta_{n,m})\leq \delta_{n,m}^{2(r-s+1)}\left(\frac{2}{\bar{g}_k}\right)^{2(r-s)}\|y\|^2.
\end{equation*}
Also, for any $x\in\mathbb{R}^{m+n}$, we have to bound
\begin{equation}
 \label{ineq1}
 \mathbb{E}\left|\left<(\tilde{R}_{B}(\eta)\Gamma)^{s-2}\tilde{R}_{B}(\eta)\Gamma\theta_l,x\right>\right|^2
\mathbbm{1}(\|\Gamma\|\leq\delta_{n,m}).
\end{equation}

In what follows, we need some additional notations. 
Let $X_1^c,\ldots,X_n^c\sim\mathcal{N}(0,\tau^2I_m)$ be the i.i.d. columns of $X$ and $(X_1^r)', \ldots, (X_n^{r})'\sim\mathcal{N}(0,\tau^2I_n)$ 
be its i.i.d. rows (here $I_m$ and $I_n$ are $m\times m$ and $n\times n$ identity matrices). 
For $j=1,\ldots,n$, define the vector $\check{X}_j^c=((X_j^c)',0)'\in {\mathbb R}^{m+n},$ 
representing the $(m+j)$-th column of matrix $\Gamma.$ Similarly, for $i=1,\ldots,m,$ 
$\check{X}_i^r=(0,(X_i^r)')'\in {\mathbb R}^{m+n}$ represents the $i$-th row of $\Gamma$.
With these notations, the following representations of $\Gamma$ holds
$$
\Gamma=\sum_{j=1}^ne_{m+j}^{m+n}\otimes\check{X}_j^c+\sum_{j=1}^n\check{X}_j^c\otimes e_{m+j}^{m+n},
$$ 
$$\Gamma=\sum_{i=1}^m\check{X}_i^r\otimes e_{i}^{m+n}+\sum_{i=1}^m e_{i}^{m+n}\otimes\check{X}_i^r,$$
and, moreover, 
$$
\sum_{j=1}^ne_{m+j}^{m+n}\otimes\check{X}_j^c=
\sum_{i=1}^m\check{X}_i^r\otimes e_{i}^{m+n},\ \ 
\sum_{j=1}^n\check{X}_j^c\otimes e_{m+j}^{m+n}=
\sum_{i=1}^m e_{i}^{m+n}\otimes\check{X}_i^r.
$$
Therefore,
\begin{equation*}
\begin{split}
 \big<(\tilde{R}_{B}(\eta)\Gamma)^{s-2}&\tilde{R}_{B}(\eta)\Gamma\theta_l,x\big>
=\sum\limits_{j=1}^n\left<\check{X}_j^c,\theta_l\right>\big<(\tilde{R}_{B}(\eta)\Gamma)^{s-2}
\tilde{R}_{B}(\eta)e_{m+j}^{m+n},x\big>\\
+&\sum\limits_{j=1}^n\left<e_{m+j}^{m+n},\theta_l\right>
\big<(\tilde{R}_{B}(\eta)\Gamma)^{s-2}\tilde{R}_{B}(\eta)\check{X}_j^c,x\big>=:I_1(x)+I_2(x),
\end{split}
\end{equation*}
and we get 
\begin{equation}
\label{bdRB}
\begin{split}
\mathbb{E}\left|\left<(\tilde{R}_{B}(\eta)\Gamma)^{s-2}
\tilde{R}_{B}(\eta)\Gamma\theta_l,x\right>\right|^2&\mathbbm{1}(\|\Gamma\|\leq\delta_{n,m})\\
\leq&2\mathbb{E}(|I_1(x)|^2+|I_2(x)|^2)\mathbbm{1}(\|\Gamma\|\leq\delta_{n,m}).
\end{split}
\end{equation}

Observe that the random variable 
$(\tilde{R}_{B}(\eta)\Gamma)^{s-2}\tilde{R}_{B}(\eta)$ 
is a function of $\{P_t\check{X}_j^c, t\neq k, j=1,\ldots,n\}.$ Indeed, since $\tilde R_B(\eta)$
is a linear combination of operators $P_t, t\neq k,$ it is easy to see that 
$(\tilde{R}_{B}(\eta)\Gamma)^{s-2}\tilde{R}_{B}(\eta)$ can be represented as a 
linear combination of operators 
$$
(P_{t_1} \Gamma P_{t_2})(P_{t_2}\Gamma P_{t_3})\dots (P_{t_{s-2}}\Gamma P_{t_{s-1}})
$$
with $t_j\neq k$ and with non-random complex coefficients. On the other hand,
$$
P_{t_k}\Gamma P_{t_{k+1}}
=\sum_{j=1}^n P_{t_k}e_{m+j}^{m+n}\otimes P_{t_{k+1}}\check{X}_j^c+
\sum_{j=1}^nP_{t_k}\check{X}_j^c\otimes P_{t_{k+1}}e_{m+j}^{m+n}.
$$
These two facts imply that $(\tilde{R}_{B}(\eta)\Gamma)^{s-2}\tilde{R}_{B}(\eta)$ is a function 
of  $\{P_t\check{X}_j^c, t\neq k, j=1,\ldots,n\}.$ Similarly, it is also a function 
of $\{P_t\check{X}_i^r, t\neq k, i=1,\ldots,m\}.$

It is easy to see that random variables $\{P_k \check X_j^c, j=1,\dots, n\}$ and $\{P_t  \check X_j^c, j=1,\dots, n, t\neq k\}$
are independent. Since they are mean zero normal random variables and $\check X_j^c, j=1,\dots, n$ are independent,
it is enough to check that, for all $j=1,\dots, n,$ $t\neq k,$ $P_k \check X_j^c$ and $P_t \check X_j^c$ are uncorrelated. 
To this end, observe that 
\begin{equation*}
\begin{split}
{\mathbb E}(P_k \check X_j^c\otimes P_t \check X_j^c)=&
P_k {\mathbb E}(\check X_j^c\otimes \check X_j^c)P_t\\
=&
\frac{1}{4}
\Big(
\begin{array}{cc}
 P_k^{uu}&P_k^{uv}\\
P_k^{vu}&P_k^{vv}
\end{array}
\Big)
\Big(
\begin{array}{cc}
I_m&0\\
0&0
\end{array}
\Big)
\Big(
\begin{array}{cc}
 P_t^{uu}&P_t^{uv}\\
P_t^{vu}&P_t^{vv}
\end{array}
\Big)
\\
=&
\frac{1}{4}\Big(
\begin{array}{cc}
P_k^{uu}P_t^{uu}&P_k^{uu}P_t^{uv}\\
P_k^{vu}P_t^{uu}&P_k^{vu}P_t^{uv}
\end{array}
\Big)=
\Big(
\begin{array}{cc}
0&0\\
0&0
\end{array}
\Big),
\end{split}
\end{equation*}
where we used orthogonality relationships (\ref{proj_uv_1}).  
Quite similarly, one can prove independence of $\{P_k\check{X}_i^r, i=1,\dots , m\}$ and $\{P_t\check{X}_i^r, i=1,\dots, m, t\neq k\}.$ 

We will now provide an upper bound on $\mathbb{E}|I_1(x)|^2\mathbbm{1}(\|\Gamma\|\leq\delta_{n,m}).$ 
To this end, define 
\begin{equation*}
\begin{split}
 \omega_j(x)=&\left<(\tilde{R}_{B}(\eta)\Gamma)^{s-2}\tilde{R}_{B}(\eta)e_{m+j}^{m+n},x\right>,\quad j=1,\ldots,n\\
 =&\omega_j^{(1)}(x)+i\omega_j^{(2)}(x)\in {\mathbb C}. 
 \end{split}
\end{equation*}
Let $I_1(x)=\kappa^{(1)}(x)+i\kappa^{(2)}(x)\in {\mathbb C}.$ Then, conditionally on $\{P_t\check{X}_j^c: t\neq k, j=1,\ldots,n\}$,
the random vector $(\kappa^{(1)}(x),\kappa^{(2)}(x))$ has the same distribution as mean zero Gaussian random vector in $\mathbb{R}^2$ with covariance,
\begin{equation*}
 \left(\sum\limits_{j=1}^n\frac{\tau^2}{2}\omega_j^{k_1}(x)\omega_j^{k_2}(x)\right), k_1,k_2=1,2
\end{equation*}
(to check the last claim, it is enough to compute conditional covariance of $(\kappa^{(1)}(x),\kappa^{(2)}(x))$ 
given $\{P_t\check{X}_j^c: t\neq k, j=1,\ldots,n\}$ using the fact that 
$(\tilde{R}_{B}(\eta)\Gamma)^{s-2}\tilde{R}_{B}(\eta)$ 
is a function of $\{P_t\check{X}_j^c, t\neq k, j=1,\ldots,n\}$). 
Therefore,
\begin{equation*}
\begin{split}
{\mathbb E}&\left(|I_1(x)|^2\Big|P_t\check{X}_j^c: t\neq k, j=1,\ldots,n\right)\\
=&
{\mathbb E}\left((\kappa^{(1)}(x))^2+ (\kappa^{(2)}(x))^2\Big|P_t\check{X}_j^c: t\neq k, j=1,\ldots,n\right)
\\
=&\frac{\tau^2}{2}\sum_{j=1}^n \Bigl((\omega_j^{(1)}(x))^2+(\omega_j^{(2)}(x))^2\Bigr)
=\frac{\tau^2}{2} \sum_{j=1}^n |\omega_j(x)|^2.
\end{split}
\end{equation*}
Furthermore, 
\begin{equation*}
 \begin{split}
\sum\limits_{j=1}^n\tau^2&|\omega_j(x)|^2= \tau^2\sum\limits_{j=1}^n|\omega_j(x)|^2\\
=&\tau^2 \sum_{j=1}^n\left|\left<\tilde{R}_{B}(\eta)(\Gamma\tilde{R}_{B}(\eta))^{s-2}x,e_{m+j}^{m+n}\right>\right|^2\\
=&\tau^2\left<\tilde{R}_{B}(\eta)(\Gamma\tilde{R}_{B}(\eta))^{s-2}x, \tilde{R}_{B}(\eta)(\Gamma\tilde{R}_{B}(\eta)\Gamma)^{s-2}x\right>\\
\leq&\tau^2\|\tilde{R}_{B}(\eta)\|^{2(s-1)}\|\Gamma\|^{2(s-2)}\|x\|^2.
 \end{split}
\end{equation*}
Under the assumption $\delta_{n,m}<\frac{\bar{g}_k}{2}$, the following inclusion holds:
$$
\{\|\Gamma\|\leq \delta_{n,m}\}\subset\left\{\sum\limits_{j=1}^n\tau^2|\omega_j(x)|^2\leq \tau^2\left(\frac{2}{\bar{g}_k}\right)^{2(s-1)}\delta_{n,m}^{2(s-2)}\|x\|^2\right\}=:G
$$
Therefore,
\begin{equation}
\label{bdI1}
 \begin{split}
  \mathbb{E}&|I_1(x)|^2\mathbbm{1}(\|\Gamma\|\leq\delta_{n,m})\leq \mathbb{E}|I_1(x)|^2\mathbbm{1}_G
  =\mathbb{E}\mathbb{E}\bigg(|I_1(x)|^2\bigg|P_t\check{X}_j^c, t\neq k, j=1,\ldots,n\bigg)\mathbbm{1}_G\\
  =&\mathbb{E}\mathbb{E}\bigg(\sum\limits_{j=1}^n\tau^2|\omega_j(x)|^2\bigg|P_t\check{X}_j^c, t\neq k, j=1,\ldots,n\bigg)\mathbbm{1}_G
  \leq \tau^2\left(\frac{2}{\bar{g}_k}\right)^{2(s-1)}\delta_{n,m}^{2(s-2)}\|x\|^2.
 \end{split}
\end{equation}
A similar bound holds also for $\mathbb{E}|I_2(x)|^2\mathbbm{1}(\|\Gamma\|\leq\delta_{n,m}):$
\begin{equation}
\label{bdI2}
 \mathbb{E}|I_2(x)|^2\mathbbm{1}(\|\Gamma\|\leq\delta_{n,m})\leq 
 \tau^2\left(\frac{2}{\bar{g}_k}\right)^{2(s-1)}\delta_{n,m}^{2(s-2)}\|x\|^2.
\end{equation}
For the proof, it is enough to observe that 
\begin{eqnarray*}
&
I_2(x)=
\sum\limits_{j=1}^n\left<e_{m+j}^{m+n},\theta_l\right>
\big<(\tilde{R}_{B}(\eta)\Gamma)^{s-2}\tilde{R}_{B}(\eta)\check{X}_j^c,x\big>
\\
&
=
\biggl\langle(\tilde{R}_{B}(\eta)\Gamma)^{s-2}\tilde{R}_{B}(\eta) \biggl(\sum_{j=1}^n\check{X}_j^c\otimes e_{m+j}^{m+n}\biggr) \theta_l, x\biggr\rangle 
\\
&
=\biggl\langle(\tilde{R}_{B}(\eta)\Gamma)^{s-2}\tilde{R}_{B}(\eta) \biggl(\sum_{i=1}^m e_{i}^{m+n}\otimes\check{X}_i^r
\biggr) \theta_l, x\biggr\rangle 
\\
&
=
\sum\limits_{i=1}^m\left<\check X_i^r,\theta_l\right>
\big<(\tilde{R}_{B}(\eta)\Gamma)^{s-2}\tilde{R}_{B}(\eta)e_i^{m+n},x\big>
\end{eqnarray*}
and to repeat the previous conditioning argument (this time, given $\{P_t\check{X}_i^r: t\neq k, i=1,\ldots,m\}$).  

Combining bounds (\ref{bdI1}), (\ref{bdI2}) and (\ref{bdRB}), we get 
$$
\mathbb{E}\left|\left<(\tilde{R}_{B}(\eta)\Gamma)^{s-2}\tilde{R}_{B}(\eta)\Gamma\theta_l,x\right>\right|^2\mathbbm{1}(\|\Gamma\|\leq\delta_{n,m})\leq 2\tau^2\left(\frac{2}{\bar{g}_k}\right)^{2(s-1)}\delta_{n,m}^{2(s-2)}\|x\|^2.
$$
Then, it follows that
\begin{equation*}
 \begin{split}
  \Big|\mathbb{E}&\left<\Gamma(R_{B}(\eta)\Gamma)^{r-s}P_k y,\theta_l\right>\left<(\tilde{R}_{B}(\eta)\Gamma)^{s-2}\tilde{R}_{B}(\eta)\Gamma\theta_l, x\right>\mathbbm{1}(\|\Gamma\|\leq\delta_{n,m})\Big|\\
  \leq&\left(\mathbb{E}\left|\left<\Gamma(R_{B}(\eta)\Gamma)^{r-s}P_k y,\theta_l\right>\right|^2\mathbbm{1}(\|\Gamma\|\leq\delta_{n,m})\right)^{1/2}\\
  \times&\left(\mathbb{E}\left|\left<(\tilde{R}_{B}(\eta)\Gamma)^{s-2}\tilde{R}_{B}(\eta)\Gamma\theta_l, x\right>\right|^2\mathbbm{1}(\|\Gamma\|\leq\delta_{n,m})\right)^{1/2}\\
  \leq&\sqrt{2}\tau\left(\frac{2\delta_{n,m}}{\bar{g}_k}\right)^{r-1}\|x\|\|y\|,
 \end{split}
\end{equation*}
which, taking into account (\ref{r-s}),  implies that
\begin{equation*}
\begin{split}
 \Big|\mathbb{E}\left<(\tilde{R}_{B}(\eta)\Gamma)^{s-1}P_k\Gamma(R_{B}(\eta)\Gamma)^{r-s}P_k y, x\right>&\mathbbm{1}(\|\Gamma\|\leq\delta_{n,m})\Big|\\
 \leq&\sqrt{2}\nu_k\tau\left(\frac{2\delta_{n,m}}{\bar{g}_k}\right)^{r-1}\|x\|\|y\|
\end{split}
 \end{equation*}
Since $(\tilde{R}_{B}(\eta)\Gamma)^rP_k=(\tilde{R}_{B}(\eta)\Gamma)^{r-1}\tilde{R}_B(\eta)\Gamma P_k$, 
it can be proved by a similar argument that
\begin{equation*}
 \left|\mathbb{E}\left<(\tilde{R}_{B}(\eta)\Gamma)^rP_k y, x\right>\mathbbm{1}(\|\Gamma\|\leq\delta_{n,m})\right|\leq \sqrt{2}\nu_k\tau\frac{2}{\bar{g}_k}\left(\frac{2\delta_{n,m}}{\bar{g}_k}\right)^{r-1}\|x\|\|y\|.
\end{equation*}
Therefore, substituting the above bounds in (\ref{rep-rep-main}) and taking into account that $|\mu_k-\eta|=\frac{\bar g_k}{2}, \eta\in 
\gamma_k$ and that the length of the contour of integration $\gamma_k$ is equal to $2\pi \frac{\bar g_k}{2},$ 
we get
\begin{equation*}
 \begin{split}
 \Big|\mathbb{E}&\left< P_k^{\perp}S_k(\Gamma)P_k y, x\right>\mathbbm{1}(\|\Gamma\|\leq\delta_{n,m})\Big|
 \leq 
 \sum\limits_{r\geq 2}\frac{r\bar{g}_k}{2}\left(\frac{2}{\bar{g}_k}\right)^2\sqrt{2}\nu_k\tau\left(\frac{2\delta_{n,m}}{\bar{g}_k}\right)^{r-1}\|x\|\|y\|\\
 =&\frac{2}{\bar{g}_k}\sqrt{2}\nu_k\tau\sum\limits_{r\geq 2}r\left(\frac{2\delta_{n,m}}{\bar{g}_k}\right)^{r-1}\|x\|\|y\|
 \lesssim_\gamma \nu_k\tau \frac{\delta_{n,m}}{\bar{g}_k^2}\|x\|\|y\|,
 \end{split}
\end{equation*}
where we also used the condition $\delta_{n,m}\leq (1-\gamma/2)\frac{\bar g_k}{2}$ implying that $\frac{2\delta_{n,m}}{\bar g_k}\leq 1-\gamma/2.$ Clearly, this implies that 
$$
\Big\|\mathbb{E} P_k^{\perp}S_k(\Gamma)P_k \Big\| \mathbbm{1}(\|\Gamma\|\leq\delta_{n,m})
\lesssim_{\gamma}  \nu_k\tau \frac{\delta_{n,m}}{\bar{g}_k^2} \lesssim_{\gamma} \frac{\nu_k \tau \sqrt{m\vee n}}{\bar g_k^2}.
$$
Furthermore, the same bound, obviously, holds for
$$
\big\|\mathbb{E}\big< P_kS_k(\Gamma)P_k^{\perp}y,x\big>\mathbbm{1}(\|\Gamma\|\leq\delta_{n,m})\big\|
=\big\|\mathbb{E} \big<P_k^{\perp}S_k(\Gamma)P_k x, y\big>\mathbbm{1}(\|\Gamma\|\leq\delta_{n,m})\big\|
$$
and, by similar arguments, it can be demonstrated that it also holds for 
\begin{equation*}
\Big\|\mathbb{E}P_k^{\perp}S_k(\Gamma)P_k^{\perp}\Big\|\mathbbm{1}(\|\Gamma\|\leq\delta_{n,m})
\end{equation*}
(the only different term in this case is $(\tilde{R}_{B}(\eta)\Gamma)^r\tilde{R}_B(\eta),$ 
but, since $\{\mu_t, t\neq k\}$ are outside of the circle $\gamma_k$, it simply leads to 
$\oint_{\gamma_{k}}(\tilde{R}_{B}(\eta)\Gamma)^r\tilde{R}_B(\eta)d\eta=0$). 

It remains to observe that 
\begin{equation*}
\begin{split}
 \Bigl\|\mathbb{E}&\left(P_k^{\perp}S_k(\Gamma)P_k+P_kS_k(\Gamma)P_k^{\perp}+P_k^{\perp}S_k(\Gamma)P_k^{\perp}\right)\mathbbm{1}(\|\Gamma\|>\delta_{n,m})\Bigr\|\\
 \leq&\mathbb{E}\Bigl\|P_k^{\perp}S_k(\Gamma)P_k+P_kS_k(\Gamma)P_k^{\perp}+P_k^{\perp}S_k(\Gamma)P_k^{\perp}\Bigr\|\mathbbm{1}(\|\Gamma\|>\delta_{n,m})\\
 \leq&\mathbb{E}\|S_k(\Gamma)\|\mathbbm{1}(\|\Gamma\|>\delta_{n,m})\\
 \leq&(\mathbb{E}\|S_k(\Gamma)\|^2)^{1/2}
 \mathbb{P}^{1/2}(\|\Gamma\|>\delta_{n,m})\\
 \lesssim&\mathbb{E}^{1/2}\left(\frac{\|\Gamma\|}{\bar{g}_k}\right)^4\mathbb{P}^{1/2}(\|\Gamma\|>\delta_{n,m})
 \lesssim \frac{1}{\sqrt{m\vee n}}\frac{\tau^2(m\vee n)}{\bar{g}_k^2}\lesssim\frac{\tau^2\sqrt{m\vee n}}{\bar{g}_k^2}
 \end{split}
\end{equation*}
and to substitute the above bounds to identity (\ref{rep-rep}) to get that   
$$
\Big\|\mathbb{E}\tilde{P}_k-P_k-P_k {\mathbb E}S_k(\Gamma) P_k\Big\|\lesssim_
{\gamma}\frac{\nu_k\tau^2 \sqrt{m\vee n}}{\bar{g}_k^2},
$$
which implies the claim of the theorem. 
\end{proof}



\begin{proof}[\bf{Proof of Theorem~\ref{thetadevthm}}]
By a simple computation (see Lemma 8 and the derivation of (6.6) in \cite{koltchinskii2014asymptotics}),
the following identity holds
\begin{equation}
\label{thetadevthmineq1}
 \begin{split}
   \big<\tilde{\theta}_{i_k}&-\sqrt{1+b_k}\theta_{i_k},x\big>=\frac{\rho_k(x)}{\sqrt{1+b_k+\rho_k(x)}}\\
 -&\frac{\sqrt{1+b_k}}{\sqrt{1+b_k+\rho_k(x)}\big(\sqrt{1+b_k+\rho_k(x)}+\sqrt{1+b_k}\big)}\rho_k(\theta_{i_k})\left<\theta_{i_k},x\right>,
 \end{split}
\end{equation}
where 
$\rho_k(x):=\big<(\tilde{P}_k-(1+b_k)P_k)\theta_{i_k},x\big>, x\in {\mathbb R}^{m+n}.$
In what follows, assume that $\|x\|=1.$
By the bounds of theorems \ref{hatPrconthm} and \ref{prdevthm},
with probability at least $1-e^{-t}:$
\begin{equation*}
 |\rho_k(x)|\leq D_{\gamma}
 \frac{\tau\sqrt{t}}{\bar{g}_k}
 \Big(\frac{\tau\sqrt{m\vee n}+\tau\sqrt{t}}{\bar{g}_k}+1\Big).
\end{equation*}
The assumption ${\mathbb E}\|X\|\leq (1-\gamma)\frac{\bar g_k}{2}$ implies that 
$\tau\sqrt{m\vee n}\lesssim \bar g_k.$ Therefore, if $t$ satisfies the assumption 
$\frac{\tau \sqrt{t}}{\bar g_k} \leq c_{\gamma}$ for a sufficiently small constant $c_{\gamma}>0,$
then we have $|\rho_k (x)|\leq \gamma/2 .$ By the assumption that $1+b_k\geq \gamma,$
this implies that $1+b_k+\rho_k(x)\geq \gamma/2.$ 
Thus, it easily follows from identity (\ref{thetadevthmineq1}) that with probability at least $1-2e^{-t}$
$$
\biggl|\big<\tilde{\theta}_{i_k}-\sqrt{1+b_k}\theta_{i_k},x\big>\biggr|\lesssim_{\gamma} 
\frac{\tau\sqrt{t}}{\bar{g}_k}
 \Big(\frac{\tau\sqrt{m\vee n}+\tau\sqrt{t}}{\bar{g}_k}+1\Big).
$$

It remains to show that the same bound holds when $\frac{\tau \sqrt{t}}{\bar g_k} >c_{\gamma}.$
In this case, we simply have that 
$$
\biggl|\big<\tilde{\theta}_{i_k}-\sqrt{1+b_k}\theta_{i_k},x\big>\biggr|\leq 
\|\tilde \theta_{i_k}\|+ (1+b_k)\|\theta_{i_k}\|\leq 2 \lesssim_{\gamma} \frac{\tau^2 t}{\bar g_k^2},
$$
which implies the bound of the theorem.

\end{proof}

\begin{proof}[\bf{Proof of Corollary~\ref{thetacor}}] 
By a simple algebra,
\begin{equation*}
 \begin{split}
 |\tilde{b}_k-b_k|=&\Big|\big<\tilde{\theta}_{i_k}^1,\tilde{\theta}_{i_k}^2\big>-(1+b_k)\Big|\leq 
 \Big|\sqrt{1+b_k}\big<\tilde{\theta}_{i_k}^1-\sqrt{1+b_k}\theta_{i_k},\theta_{i_k}\big>\Big|\\
   +&\Big|\sqrt{1+b_k}\big<\tilde{\theta}_{i_k}^2-\sqrt{1+b_k}\theta_{i_k},\theta_{i_k}\big>\Big|+\Big|\big<\tilde{\theta}_{i_k}^1-\sqrt{1+b_k}\theta_{i_k},\tilde{\theta}_{i_k}^2-\sqrt{1+b_k}\theta_{i_k}\big>\Big|.
\end{split}   
\end{equation*}
Corollary~\ref{thetacor}  implies that with probability at least $1-e^{-t}$
\begin{equation*}
\Big|\sqrt{1+b_k}\big<\tilde{\theta}_{i_k}^1-\sqrt{1+b_k}\theta_{i_k},\theta_{i_k}\big>\Big|  
\lesssim_{\gamma} \frac{\tau\sqrt{t}}{\bar{g}_k}\Big[\frac{\tau\sqrt{m\vee n}+\tau\sqrt{t}}{\bar{g}_k}+1\Big],
 \end{equation*}
 where we also used the fact that $1+b_k\in [0,1].$ A similar bound holds with the same probability 
 for 
 $$
 \Big|\sqrt{1+b_k}\big<\tilde{\theta}_{i_k}^2-\sqrt{1+b_k}\theta_{i_k},\theta_{i_k}\big>\Big|.   
$$ 
To control the remaining term 
 $$
 \Big|\big<\tilde{\theta}_{i_k}^1-\sqrt{1+b_k}\theta_{i_k},\tilde{\theta}_{i_k}^2-\sqrt{1+b_k}\theta_{i_k}\big>\Big|, 
 $$
 note that $\tilde{\theta}_{i_k}^1$ and $\tilde{\theta}_{i_k}^2$ are independent. Thus, applying the bound of 
 Theorem~\ref{thetadevthm} conditionally on  $\tilde{\theta}_{i_k}^2,$ we get that with probability at least $1-e^{-t}$
 $$
 \Big|\big<\tilde{\theta}_{i_k}^1-\sqrt{1+b_k}\theta_{i_k},\tilde{\theta}_{i_k}^2-\sqrt{1+b_k}\theta_{i_k}\big>\Big|
 \lesssim_{\gamma} \frac{\tau\sqrt{t}}{\bar{g}_k}\Big[\frac{\tau\sqrt{m\vee n}+\tau\sqrt{t}}{\bar{g}_k}+1\Big]
 \|\tilde{\theta}_{i_k}^2-\sqrt{1+b_k}\theta_{i_k}\|.
 $$
 It remains to observe that 
\begin{equation*}
 \|\tilde{\theta}_{i_k}^2-\sqrt{1+b_k}\theta_{i_k}\|\leq 2
\end{equation*}
to complete the proof of bound (\ref{bkbd}).

Assume that $\|x\|\leq 1.$
Recall that under the assumptions of the corollary, $\tau\sqrt{m\vee n}\lesssim_{\gamma}\bar g_k$ and,
if $\frac{\tau\sqrt{t}}{\bar{g}_k}\leq c_{\gamma}$ for a sufficiently small constant $c_{\gamma,}$
then bound (\ref{bkbd}) implies that $|\tilde b_k-b_k|\leq \gamma/4$ (on the event of probability at least $1-e^{-t}$). Since $1+b_k\geq \gamma/2,$ on the same event we also have $1+\tilde b_k\geq \gamma/4$ implying that $\hat \theta_{i_k}^{(\gamma)}=\frac{\tilde \theta_{i_k}^1}{\sqrt{1+\tilde b_k}}.$ 
Therefore, 
\begin{eqnarray}
\label{last_last}
&
\Big|\big<\hat{\theta}_{i_k}^{(\gamma)}-\theta_{i_k},x\big>\Big|= \frac{1}{\sqrt{1+\tilde b_k}}
\Big|\big<\tilde{\theta}^1_{i_k}-\sqrt{1+\tilde b_k}\theta_{i_k},x\big>\Big|
\\
\nonumber
&
\lesssim_{\gamma}\Big|\big<\tilde{\theta}^1_{i_k}-\sqrt{1+b_k}\theta_{i_k},x\big>\Big|+ \Big|\sqrt{1+b_k}-\sqrt{1+\tilde b_k}\Big|.
\end{eqnarray}
The first term in the right hand side can be bounded using Theorem~\ref{thetadevthm} and, for the second term,
$$
 \Big|\sqrt{1+b_k}-\sqrt{1+\tilde b_k}\Big|= \frac{|\tilde b_k-b_k|}{\sqrt{1+b_k}+\sqrt{1+\tilde b_k}}
 \lesssim_{\gamma} |\tilde b_k-b_k|,
 $$ 
so bound (\ref{bkbd}) can be used. Substituting these bounds in (\ref{last_last}), we derive (\ref{estbd}) in the 
case when $\frac{\tau\sqrt{t}}{\bar{g}_k}\leq c_{\gamma}.$ 

In the opposite case, when  $\frac{\tau\sqrt{t}}{\bar{g}_k}> c_{\gamma},$ we have 
$$
\Big|\big<\hat{\theta}_{i_k}^{(\gamma)}-\theta_{i_k},x\big>\Big|\leq \|\hat{\theta}_{i_k}^{(\gamma)}\|+\|\theta_{i_k}\|
\leq \frac{1}{\sqrt{1+\tilde b_k}\vee \frac{\sqrt{\gamma}}{2}} +1 \leq \frac{2}{\sqrt{\gamma}}+1.
$$
Therefore,
$$
\Big|\big<\hat{\theta}_{i_k}^{(\gamma)}-\theta_{i_k},x\big>\Big|\lesssim_{\gamma}
\frac{\tau \sqrt{t}}{\bar g_k},
$$
which implies (\ref{estbd}) in this case. 

\end{proof}

\bibliographystyle{abbrv}
\bibliography{refer}

\begin{thebibliography}{10}

\bibitem{benaych2012singular}
F.~Benaych-Georges and R.~R. Nadakuditi.
\newblock The singular values and vectors of low rank perturbations of large
  rectangular random matrices.
\newblock {\em Journal of Multivariate Analysis}, 111:120--135, 2012.

\bibitem{eisenstat1994relative}
S.~C. Eisenstat and I.~C. Ipsen.
\newblock Relative perturbation bounds for eigenspaces and singular vector
  subspaces.
\newblock In {\em Proceedings of the Fifth SIAM Conference on Applied Linear
  Algebra}, pages 62--66, 1994.

\bibitem{huang2009spectral}
L.~Huang, D.~Yan, N.~Taft, and M.~I. Jordan.
\newblock Spectral clustering with perturbed data.
\newblock In {\em Advances in Neural Information Processing Systems}, pages
  705--712, 2009.

\bibitem{jin2015fast}
J.~Jin.
\newblock Fast community detection by {SCORE}.
\newblock {\em The Annals of Statistics}, 43(1):57--89, 2015.

\bibitem{kannan2009spectral}
R.~Kannan and S.~Vempala.
\newblock {\em Spectral algorithms}.
\newblock Now Publishers Inc, 2009.

\bibitem{koltchinskii2014asymptotics}
V.~Koltchinskii and K.~Lounici.
\newblock Asymptotics and concentration bounds for spectral projectors of
  sample covariance.
\newblock {\em arXiv preprint arXiv:1408.4643}, 2014.

\bibitem{lei2014consistency}
J.~Lei and A.~Rinaldo.
\newblock Consistency of spectral clustering in stochastic block models.
\newblock {\em The Annals of Statistics}, 43(1):215--237, 2014.

\bibitem{li1998relative}
R.-C. Li.
\newblock Relative perturbation theory: Eigenspace and singular subspace
  variations.
\newblock {\em SIAM Journal on Matrix Analysis and Applications},
  20(2):471--492, 1998.

\bibitem{o2013random}
S.~O'Rourke, V.~Vu, and K.~Wang.
\newblock Random perturbation of low rank matrices: Improving classical bounds.
\newblock {\em arXiv preprint arXiv:1311.2657}, 2013.

\bibitem{rohe2011spectral}
K.~Rohe, S.~Chatterjee, and B.~Yu.
\newblock Spectral clustering and the high-dimensional stochastic block model.
\newblock {\em The Annals of Statistics}, 39(4):1878--1915, 2011.

\bibitem{rudelson2013delocalization}
M.~Rudelson and R.~Vershynin.
\newblock Delocalization of eigenvectors of random matrices with independent
  entries.
\newblock {\em arXiv preprint arXiv:1306.2887}, 2013.

\bibitem{stewart1998perturbation}
G.~W. Stewart.
\newblock Perturbation theory for the singular value decomposition.
\newblock 1998.

\bibitem{vershynin2010introduction}
R.~Vershynin.
\newblock Introduction to the non-asymptotic analysis of random matrices.
\newblock {\em arXiv preprint arXiv:1011.3027}, 2010.

\bibitem{vu2011singular}
V.~Vu.
\newblock Singular vectors under random perturbation.
\newblock {\em Random Structures \& Algorithms}, 39(4):526--538, 2011.

\bibitem{vu2014random}
V.~Vu and K.~Wang.
\newblock Random weighted projections, random quadratic forms and random
  eigenvectors.
\newblock {\em Random Structures \& Algorithms}, 2014.

\bibitem{wang2012singular}
R.~Wang.
\newblock Singular vector perturbation under {Gaussian} noise.
\newblock {\em SIAM Journal on Matrix Analysis and Applications},
  36(1):158--177, 2015.

\end{thebibliography}
\end{document}